\documentclass[reqno]{amsart}
\usepackage{amssymb,latexsym,amsmath,amsthm,enumerate,amsbsy}
\usepackage[mathscr]{eucal}
\usepackage{framed,color,graphicx}
\usepackage{mathrsfs}
\usepackage[all]{xy}
\usepackage{tikz}
\usepackage{cite}

\makeatletter
\@namedef{subjclassname@2020}{\textup{2020} Mathematics Subject Classification}
\makeatother

\usetikzlibrary{positioning,decorations.pathreplacing,patterns,decorations.pathmorphing}
\tikzset{%
element/.style={draw, shape=circle, fill=white, inner sep=1.4pt}
}

\DeclareSymbolFont{bbold}{U}{bbold}{m}{n}
\DeclareSymbolFontAlphabet{\mathbbold}{bbold}

\theoremstyle{plain}
\newtheorem{theorem}{Theorem}[section]
\newtheorem{lemma}[theorem]{Lemma}
\newtheorem{corollary}[theorem]{Corollary}
\newtheorem{proposition}[theorem]{Proposition}

\newtheorem{problem}[theorem]{Problem}
\newtheorem{conjecture}[theorem]{Conjecture}

\newtheorem{definition}[theorem]{Definition}

\newtheorem{remark}[theorem]{Remark}

\newcommand{\ba}{\mathbf{a}}

\newcommand{\bc}{\mathbf{c}}

\newcommand{\bp}{\mathbf{p}}
\newcommand{\bq}{\mathbf{q}}

\newcommand{\bt}{\mathbf{t}}
\newcommand{\bu}{\mathbf{u}}
\newcommand{\bv}{\mathbf{v}}
\newcommand{\bw}{\mathbf{w}}

\begin{document}

\title[The finite basis problem for the flat semirings $S(W)$]
{The finite basis problem for the flat semirings $S(W)$}

\author{Zidong Gao}
\address{School of Mathematics, Northwest University, Xi'an, 710127, Shaanxi, P.R. China}
\email{zidonggao@yeah.net}

\author{Miaomiao Ren}
\address{School of Mathematics, Northwest University, Xi'an, 710127, Shaanxi, P.R. China}
\email{miaomiaoren@yeah.net}

\author{Xianzhong Zhao}
\address{School of Mathematics and Data Science, Shaanxi University of Science and Technology,
Xi'an, 710021, Shaanxi, P.R. China}
\email{zhaoxz@nwu.edu.cn}


\begin{abstract}
We focus on the finite basis problem for flat semirings of the form $S(W)$,
where $W$ is an arbitrary set of nonempty words.
We prove that $S(W)$ generates a Cross variety (and hence is finitely based)
whenever every word in $W$ has length at most $3$,
whereas it is nonfinitely based whenever there exists $k \geq 3$ such that
$W$ is $x^{k+2}$-free but not $x^{k+1}$-free.
In particular, if $W_k$ denotes the set of all words of length $k$,
then $S(W_k)$ is finitely based if and only if $k \leq 3$.
Moreover, $S(W)$ is nonfinitely based whenever $W$ is finite and not $x^4$-free.
These results provide a partial answer to an open problem raised by Jackson et al.~(J Algebra 611: 211--245, 2022).
\end{abstract}

\maketitle
\section{Introduction}\label{sec:intro}
An \emph{additively idempotent semiring} (or ai-semiring for short)
is an algebra $(S, +, \cdot)$ of type $(2,2)$ such that the additive reduct $(S, +)$ is a commutative idempotent semigroup,
the multiplicative reduct $(S, \cdot)$ is a semigroup, and the distributive laws
\[
x(y+z) \approx xy + xz, \quad (x+y)z \approx xz + yz
\]
hold.

The additive reduct of an ai-semiring $S$ is a semilattice, and the natural partial order $\leq$ on $S$ defined by
\[
a \leq b\;\Leftrightarrow\; a+b=b
\]
is compatible with addition and multiplication.
Consequently, an ai-semiring is often called a \emph{semilattice-ordered semigroup} (see, e.g., \cite{kp}).
Unless otherwise stated, any order-theoretic statement concerning an ai-semiring refers to this order.

For an ai-semiring $S$, if $J$ is both a multiplicative ideal and an additive order filter of $S$,
we may form the ideal quotient $S/J$ by collapsing all elements of $J$ to a single element,
while leaving all elements outside $J$ as distinct singleton classes.
In this quotient algebra, $J$ serves as the multiplicative zero and the additive maximum element.

The class of ai-semirings (possibly with extra unary operations or constants)
includes the Kleene semiring of regular languages~\cite{con}, the max-plus algebra~\cite{aei},
the semiring of all binary relations on a set~\cite{dol09}, the matrix semiring over an ai-semiring~\cite{bg},
and distributive lattices~\cite{bs}.
These and other similar algebras have played important roles
in several branches of mathematics, including algebraic geometry~\cite{cc}, tropical geometry~\cite{ms}, information science~\cite{gl}, and theoretical computer science~\cite{go}.

A class of ai-semirings is a \emph{variety} if it is closed under taking subalgebras, homomorphic images,
and arbitrary direct products.
By Birkhoff's theorem, a class of ai-semirings is a variety if and only if it is an \emph{equational class};
that is, the class of all ai-semirings satisfying a certain set of identities.
Let $\mathcal{V}$ be an ai-semiring variety.
Then $\mathcal{V}$ is \emph{finitely based} if it can be defined by a finite set of identities; otherwise,
it is \emph{nonfinitely based}.

For any class $\mathcal{K}$ of ai-semirings, let $\mathsf{V}(\mathcal{K})$ denote the variety generated by $\mathcal{K}$;
that is, the smallest variety containing $\mathcal{K}$.
When $\mathcal{K} = \{A_i \mid i \in I\}$, we usually write $\mathsf{V}(\mathcal{K}) = \mathsf{V}(A_i\mid i \in I)$.
It is well known that $\mathsf{V}(\mathcal{K})$ consists of all
homomorphic images of subalgebras of direct products of algebras in $\mathcal{K}$.
An ai-semiring $S$ is finitely based (nonfinitely based) if the variety $\mathsf{V}(S)$ is finitely based (nonfinitely based).

The finite basis problem for a class of ai-semirings, one of the most important problems in universal algebra,
concerns the classification of its members according to whether they are finitely based.
In the last two decades, the finite basis problem for ai-semirings has been intensively studied,
and considerable progress has been made;
see, for example,~\cite{dol07, dol09, dgv25, gjrz, gpz05, jrz, pas05, rlzc, rlyc, rjzl, rzw, sr, shap23, vol21, wrz, wzr, yrzs, zrc}.
For more information on the finite basis problem of other related algebras such as semigroups and monoids,
see the monograph~\cite{lee23}.

Dolinka~\cite{dol07} found the first example of a nonfinitely based finite ai-semiring.
Jackson~\cite{jac:flat} solved the finite basis problem for the flat extensions of finite groups,
thereby providing infinitely many nonfinitely based finite ai-semirings.
Shao and Ren~\cite{sr} proved that every ai-semiring in the variety generated by all ai-semirings of order two is finitely based.
Jackson et al.~\cite{jrz} and Zhao et al.~\cite{zrc} classified three-element ai-semirings with respect to the finite basis property,
showing that $S_7$ is the unique nonfinitely based three-element ai-semiring.
The Cayley tables of $S_7$ are given in Table~\ref{tb24111401}.

\begin{table}[ht]
\caption{The Cayley tables of $S_7$} \label{tb24111401}
\begin{tabular}{c|ccc}
$+$      &$0$&$a$&$1$\\
\hline
$0$ &$0$&$0$&$0$\\
$a$      &$0$&$a$&$0$\\
$1$      &$0$&$0$&$1$\\
\end{tabular}\qquad
\begin{tabular}{c|ccc}
$\cdot$  &$0$&$a$&$1$\\
\hline
$0$ &$0$&$0$&$0$\\
$a$      &$0$&$0$&$a$\\
$1$      &$0$&$a$&$1$\\
\end{tabular}
\end{table}

Moreover, Volkov~\cite{vol21} and Jackson et al.~\cite{jrz} independently
resolved the finite basis problem for the ai-semiring whose multiplicative reduct is the six-element Brandt monoid.
Ren et al.~\cite{rjzl} provided an infinite number of minimal nonfinitely based ai-semiring varieties.
Very recently, the finite basis problem for four-element ai-semirings has been nearly completed, with only two algebras remaining (see \cite{yr2602, yrg, rlyc, rlzc, ryy, yrzs}).

In much of the above work, flat semirings have played an important (sometimes even a decisive) role.
We now recall their definition.
By a \emph{flat semiring} we mean an ai-semiring
whose multiplicative reduct has a zero element $0$ and satisfies $a+b=0$ for all distinct $a,b\in S$.
The three-element algebra $S_7$ is a typical example of a flat semiring.
As observed in~\cite[Lemma 2.2]{jrz} (see also~\cite[Lemma 4.1.1]{ek}),
a semigroup with zero \(0\) becomes a flat semiring if and only if it is \(0\)-cancellative; that is, for all $a, b, c \in S$,
\[
ab = ac \neq 0 \;\Rightarrow\; b = c \quad\text{and}\quad ab = cb \neq 0 \;\Rightarrow\; a = c.
\]

Let $\mathbf{F}$ denote the variety generated by all flat semirings. The following result, due to Jackson et al.~\cite[Lemma 2.1]{jrz}, solves the finite basis problem for $\mathbf{F}$ and characterizes its subdirectly irreducible members.

\begin{lemma}\label{lem24121301}
The variety $\mathbf{F}$ is finitely based, and every subdirectly irreducible member of $\mathbf{F}$ is a flat semiring.
\end{lemma}

For any cancellative semigroup $(S, \cdot)$,
let $S^0$ denote the semigroup obtained from $S$ by adjoining a multiplicative zero $0$.
Then $S^0$ is $0$-cancellative, and therefore becomes a flat semiring;
we call it the \emph{flat extension} of $S$ and denote it by $\flat(S)$.

Groups, being cancellative, provide a natural source of examples via their flat extensions.
The finite basis problem for the flat extensions of finite groups has been completely solved by Jackson~\cite[Theorem 7.3]{jac:flat}:
if $G$ is a finite group, then the flat extension $\flat(G)$ is finitely based if and only if all Sylow subgroups of $G$ are abelian. Equivalently, $\flat(G)$ is nonfinitely based if and only if $G$ contains a nonabelian nilpotent subgroup.

Subsequently, Jackson et al.~\cite[Theorem 6.1]{jrz} generalized this result: if the multiplicative reduct of a finite ai-semiring $S$ contains a nonabelian nilpotent subgroup, then every finite ai-semiring whose variety contains $S$ is nonfinitely based.

Next, we introduce another important class of flat semirings.
Let $W$ be a set of words in the free semigroup $X^+$ over $X$,
and let $W^{\leq}$ denote the set of all nonempty subwords of words in $W$.
Define a multiplication on $W^{\leq} \cup \{0\}$ by
\[
0 \cdot 0 = 0 \cdot \mathbf{u} = \mathbf{u} \cdot 0 = 0
\]
for all $\mathbf{u} \in W^{\leq}$, and for $\mathbf{u}, \mathbf{v} \in W^{\leq}$,
\[
\mathbf{u} \cdot \mathbf{v} =
\begin{cases}
\mathbf{u}\mathbf{v}, & \text{if } \mathbf{u}\mathbf{v} \in W^{\leq}, \\
0, & \text{otherwise}.
\end{cases}
\]
The resulting algebra is a semigroup and is denoted by $S(W)$.
If the empty word $1$ is allowed in the above construction,
then the corresponding semigroup becomes a monoid, which we denote by $M(W)$.
Working in the free commutative semigroup $X_c^+$ (or monoid $X_c^*$) yields the commutative analogues $S_c(W)$ and $M_c(W)$.
For $W = \{\mathbf{u}_i \mid i \in I\}$, we write $S(\mathbf{u}_i \mid i \in I)$ in place of $S(\{\mathbf{u}_i \mid i \in I\})$
to avoid nested braces, and similarly for $M(W)$, $S_c(W)$ and $M_c(W)$.

All of the above semigroups are $0$-cancellative,
and therefore become flat semirings (which we denote by the same symbols).
Moreover, these flat semirings can be realized as ideal quotients of flat extensions of free semigroups or free monoids,
as we now illustrate.

Suppose that $S$ is an ai-semiring containing a proper nonempty subset $J$
that is both a multiplicative ideal and an additive order filter.
Then $J$ is called an \emph{ideal} of $S$, and we can form the \emph{ideal quotient} $S/J$
by collapsing all elements of $J$ to a single element,
while leaving all elements outside $J$ as distinct singleton classes.
For example, the flat semiring \(S(W)\) is isomorphic to the ideal quotient of \(\flat(X^+)\) with respect to the ideal \(\flat(X^+)\backslash W^{\leq}\).

Jackson et al.~\cite{jrz} proved that every finite flat semiring of the form $M(W)$ or $M_c(W)$ is nonfinitely based.
They also proposed the following problem~\cite[Problem 7.1]{jrz}:

\begin{problem}\leavevmode
\rm
\begin{enumerate}[$(1)$]
\item For which finite sets of words $W$ is $S(W)$ finitely based as an ai-semiring?

\item When $W$ consists of powers of letters and words of length at most $2$, under what condition is $S_c(W)$ finitely based?

\item More generally than in parts (1) and (2): If a finite nilpotent semigroup carries a flat semiring structure, under what conditions is it finitely based?

\item Even more generally, which finite flat semirings are finitely based?
\end{enumerate}
\end{problem}

Recent progress has been made on this problem.
Regarding part (2), it follows from \cite{jrz, wzr} that for a finite set $W$ of words,
$S_c(W)$ is finitely based if and only if every word in $W$
is either a cube of a letter or a word of length at most $2$.
This completes the classification of finite flat semirings of the form $S_c(W)$
with respect to the finite basis property.

For part (3), Gao and Ren~\cite{gr} showed that every finite flat semiring
whose multiplicative reduct is $3$-nilpotent is nonfinitely based.
In contrast, at the 2024 Conference on Theoretical and Computational Algebra,
Marcel Jackson demonstrated that the finite basis property for finite flat semirings
whose multiplicative reduct is $4$-nilpotent is \texttt{NP}-hard.

Concerning part (4), some progress has also been made.
For instance, Jackson et al.~\cite{jrz} showed that every finite flat semiring whose variety contains $S_7$ is nonfinitely based.
Gao et al.~\cite{gjrz} completely solved the finite basis problem for flat semirings in the variety $\mathsf{V}(S_7)$.
Moreover, Ren et al.~\cite{rlzc} and Shaprynski\v{\i}~\cite{shap23} settled the finite basis problem for all four-element flat semirings.

The present paper focuses on part (1).
While parts (2)--(4) have seen substantial progress over the past few years,
part (1) has remained largely untouched.
Indeed, despite its seemingly modest formulation,
part (1) turns out to be considerably more difficult than the others,
owing to the combinatorial complexity of arbitrary word sets.
To the best of our knowledge,
only two previous results are known for part (1).
Wu et al.~\cite[Corollary 4.4]{wzr} established a sufficient condition
under which a finite flat semiring $S(W)$ is nonfinitely based in a restricted setting,
while \cite[Theorem 5.5]{rjzl} provided a family of nonfinitely based infinite flat semirings of the form $S(W)$.
The present paper takes the first substantial step towards a complete resolution of part (1).
We investigate the finite basis problem for flat semirings of the form $S(W)$, where $W$ is not necessarily finite.

The paper is organized as follows.
Section~\ref{sec:prelim} collects the necessary preliminaries.
Section~\ref{sec:2} establishes the basic properties of flat semirings of the form $S(W)$.
In Section~\ref{sec:w3},
we prove that the variety $\mathcal{W}_3 := \mathsf{V}(S(W_3))$ is a Cross variety,
where $W_3$ denotes the set of all words in $X^+$ of length $3$;
as a consequence, $S(W)$ is finitely based whenever every word in $W$ has length at most $3$.
Finally, Section~\ref{sec:NFB} provides a sufficient condition for non-finite basis and applies
it to show that certain flat semirings $S(W)$ are nonfinitely based.

\section{Preliminaries}\label{sec:prelim}
In this section, we collect some basic notions, notation, and tools that will be
used throughout the paper.

Each element of the free semigroups $X^+$ is called a \emph{word} or a \emph{monomial}, written as bold lowercase letters $\mathbf{u}, \mathbf{v}, \mathbf{w},\ldots$, while ordinary lowercase letters $x,y,z,\ldots$ stand for variables.

From \cite[Theorem 2.5]{kp} we know that
the set $P_f(X^+)$ of all nonempty finite subsets of $X^+$,
equipped with addition as set-theoretic union and multiplication as elementwise product,
is the free ai-semiring over $X$.
Consequently, an \emph{ai-semiring term} (or simply a \emph{term}) is precisely a nonempty finite subset of $X^+$.
In keeping with standard semiring notation,
we write the term $\{\mathbf{u}_1, \mathbf{u}_2, \dots, \mathbf{u}_n\}$
as the \emph{polynomial} $\mathbf{u}_1 + \mathbf{u}_2 + \cdots + \mathbf{u}_n$.

Now let $\mathbf{w}$ be a word, let $\bu=\mathbf{u}_1 + \mathbf{u}_2 + \cdots + \mathbf{u}_n$ a polynomial. Then
\begin{itemize}
\item  $c(\mathbf{w})$ denotes the \emph{content} of $\mathbf{w}$; that is, the set of all variables occurring in $\mathbf{w}$.

\item $\ell(\mathbf{w})$ denotes the \emph{length} of $\mathbf{w}$; that is, the number of variables occurring in $\mathbf{w}$ counting multiplicities.

\item $c(\mathbf{u}) = \bigcup_{1\leq i\leq n} c(\mathbf{u}_i)$.

\item $\bu(x_1,\ldots,x_n)$ denotes a polynomial with $c(\bu)\subseteq \{x_1, \ldots, x_n\}$.

\end{itemize}

An \emph{ai-semiring identity} (or simply an \emph{identity}) is
a formal expression of the form $\bu \approx \bv$, where $\bu$ and $\bv$ are polynomials.
We say that an ai-semiring $S$ \emph{satisfies} an identity $\bu \approx \bv$ over $X=\{x_1,\ldots,x_n\}$ if
\[
\bu(a_1, a_2, \dots, a_n) = \bv(a_1, a_2, \dots, a_n)
\]
for all $a_1, a_2, \dots, a_n \in S$.
Equivalently, $\varphi(\bu) = \varphi(\bv)$ for every semiring homomorphism $\varphi \colon P_f(X^+) \to S$.
Such a homomorphism is also called an \emph{assignment},
and is uniquely determined by the images of the elements of $X$;
for convenience, we may denote it simply by $\varphi \colon X \to S$.


Suppose that $\bu = \bu_1 + \bu_2 + \cdots + \bu_m$ and $\bv = \bv_1 + \bv_2 + \cdots + \bv_n$.
Then it is easy to verify that $S$ satisfies the identity $\bu \approx \bv$ if and only if
it satisfies the identities $\bu \approx \bu + \bv_i$ and $\bv \approx \bv + \bu_j$ for all $1 \leq i \leq n$, $1 \leq j \leq m$.
Consequently, $S$ lies in an ai-semiring variety $\mathcal{V}$  if and only if
it satisfies every nontrivial identity of $\mathcal{V}$ of the form $\bp \approx \bp + \bq$,
where $\bp$ is a polynomial and $\bq$ is a word.


For convenience, we abbreviate the following identities
\begin{equation}\label{Ufree}
\bu z \approx z\bu \approx \bu+z \approx \bu,
\end{equation}
where $z \notin c(\bu)$, as $\bu \approx 0$. We call this abbreviation the \emph{$\bu$-free identity}.
If an ai-semiring $S$ satisfies $\bu\approx 0$, then for any assignment $\varphi$,
the element $\varphi(\bu)$ is both a multiplicative zero and an additive maximum element in $S$, which justifies the abbreviation.

Let $S$ be a semigroup with zero element $0$. Then

\begin{itemize}
\item $S$ is \emph{nil} if for every $a \in S$ there exists $k\geq 1$ such that $a^k = 0$.

\item $S$ is \emph{$k$-nil} if there exists $k\geq 1$ such that $a^k = 0$ for all $a \in S$.


\item $S$ is \emph{$k$-nilpotent} if $S^k = \{0\}$ for some $k \geq 1$.

\item $S$ is \emph{nilpotent} if it is $k$-nilpotent for some $k \geq 1$.

\end{itemize}
An ai-semiring is called \emph{nil} (respectively, \emph{$k$-nil}, \emph{nilpotent}, \emph{$k$-nilpotent}) if its multiplicative reduct satisfies the corresponding property.

\begin{remark}\label{finnil}
Every nilpotent semigroup is nil, but the converse is not true in general. However, every finite nil semigroup is nilpotent \rm (see  \cite[Lemma~3.1.16]{sapir14}).
\end{remark}

Let $\bu$ and $\bv$ be words in $X^+$.
We say that $\bv$ is \emph{$\bu$-free} if $\varphi(\bu)$ is not a subword of $\bv$ for any semigroup homomorphism $\varphi\colon X^+\to X^+$.
In the same spirit,
a set $W$ of words in $X^+$ is \emph{$\bu$-free} if every word in $W$ is $\bu$-free.
It is easy to see that $W$ is not $\bu$-free if and only if
$\varphi(\bu)\in W^{\leq}$ for some semigroup homomorphism $\varphi\colon X^+\to X^+$.
Moreover, the flat semiring $S(W)$ satisfies the $\bu$-free identity $\bu \approx 0$ if and only if $W$ is $\bu$-free.
In particular, $S(W)$ is $k$-nilpotent if and only if $W$ is $x_1\cdots x_k$-free
(i.e., every word in $W$ has length less than $k$),
and $S(W)$ is $k$-nil if and only if $W$ is $x^k$-free.

\section{The flat semiring $S(W)$}\label{sec:2}
In this section,
we establish the basic properties of flat semirings of the form $S(W)$.
Some of the material below is drawn from \cite{jrz,gr};
for the reader's convenience and to keep the paper self-contained,
we summarize the parts that will be used.

Recall that a nontrivial algebra is \emph{subdirectly irreducible} if it has a least non-diagonal congruence.
It is easy to see that there is a one-to-one, inclusion-preserving
correspondence between semiring congruences and multiplicative ideals on a flat semiring.
This correspondence hinges on the fact that, in a flat semiring,
every multiplicative ideal naturally determines a semiring congruence, since it is an additive filter.
Conversely, every semiring congruence on a flat semiring with zero $0$
has its $0$-congruence class as a multiplicative ideal.
This yields the desired one-to-one inclusion-preserving correspondence.
Consequently, a flat semiring is subdirectly irreducible if and only if it has a least nonzero multiplicative ideal.

By \cite[Proposition 3.17]{b}, the following proposition follows readily.

\begin{proposition}\label{siflat}
Let $(S_i)_{i\in I}$ be a family of flat semirings and let $S$ be a flat semiring.
Then $S$ is isomorphic to a subdirect product of $(S_i)_{i\in I}$ if and only if
there exists a family $(J_i)_{i\in I}$ of ideals of $S$ such that $S/J_i$ is isomorphic to $S_i$
for each $i\in I$ and $\bigcap_{i\in I} J_i = \{0\}$.
\end{proposition}

A nonzero element $\omega$ of a flat semiring $S$ is an \emph{annihilator}
if $\omega s = s\omega = 0$ for all $s \in S$.
From the definition it follows directly that the annihilators of $S(W)$ are precisely the maximal words in $W$ under the subword order.
The following result characterizes subdirectly irreducible nil flat semirings in terms of annihilators.

\begin{lemma}[{\cite[Proposition 1.4]{gr}}]\label{nilsi}
A nil flat semiring $S$ is subdirectly irreducible if and only if
it has a unique annihilator $\omega$ and $\omega\in S^{1}aS^{1}$ for all $a\in S\backslash\{0\}$.
In this case, $\{0,\omega\}$ is the least non-zero multiplicative ideal of $S$.
\end{lemma}

\begin{lemma}[{\cite[Proposition 1.5]{gr}}]\label{nilpotentsi}
A nilpotent flat semiring is subdirectly irreducible if and only if it has a unique annihilator $\omega$.
In this case, $\{0,\omega\}$ is its least nonzero multiplicative ideal.
\end{lemma}

\begin{lemma}[{\cite[Lemma 1.7]{gr}}]\label{sk0w}
Let $S$ be a nilpotent flat semiring containing a unique annihilator $\omega$.
If $S^{k+1} = \{0\}$ but $S^{k} \neq \{0\}$ for some $k\geq 1$, then $S^{k} = \{0, \omega\}$.
\end{lemma}

\begin{lemma}[{\cite[Proposition 1.10]{gr}}]\label{finsubf}
Every subvariety of $\mathbf{F}$ can be generated by a single flat semiring.
In particular, every finitely generated subvariety of $\mathbf F$ can be generated by a finite flat semiring.
\end{lemma}

The following construction, based on the least non-zero ideal $\{0,\omega\}$ in Lemma~\ref{nilsi}, will be useful in the sequel.

\begin{definition}[{\cite[Definition 1.11]{gr}}]
A flat semiring $S$ is the \emph{$\{0,\omega\}$-direct union} of a family $(S_i)_{i\in I}$ of subdirectly irreducible flat nil-semirings,
denoted by $\bigcup_{i\in I}^{\omega} S_i$,
if there exist a family $(S_i')_{i\in I}$ of subsemirings of $S$ such that
\[
S = \bigcup_{i\in I} S_i',
\quad
S_i' \cong S_i,
\quad
S_j' \cap S_k' = \{0, \omega\},
\quad
S_j' \cdot S_k' = \{0\}
\]
for all $i, j, k \in I$ with $j \neq k$.
\end{definition}

\begin{remark}
The above definition itself indicates how to construct the $\{0, \omega\}$-direct union for any family $(S_i)_{i\in I}$ of subdirectly irreducible nil flat semirings. In the sequel we sometimes denote by $S_1 \circ S_2$ the $\{0, \omega\}$-direct union of $S_1$ and $S_2$.
\end{remark}

\begin{lemma}[{\cite[Proposition 1.12]{gr}}]
Suppose that $\{S_i\}_{i\in I}$ is a family of subdirectly irreducible flat nil-semirings with pairwise intersection $\{0,\omega\}$.
If $J$ is a subset of $I$ and $T_i$ is a subalgebra of $S_i$ containing $\omega$ for each $i\in J$,
then $\bigcup_{i\in J}^{\omega} T_i$ is a subalgebra of $\bigcup_{i\in I}^{\omega} S_i$.
\end{lemma}

\begin{lemma}[{\cite[Proposition 1.14]{gr}}]\label{0wunion}
The $\{0,\omega\}$-direct union of a family of subdirectly irreducible
flat nil-semirings is also subdirectly irreducible.
\end{lemma}

The following proposition characterizes when $S(W)$ is subdirectly irreducible.

\begin{proposition}\label{swsi}
Let $W$ be a nonempty set of words in $X^+$.
Then the flat semiring $S(W)$ is subdirectly irreducible if and only if $S(W) = S(\bw)$ for some word $\bw\in W$.
In this case, $S(W)$ is necessarily finite and nilpotent.
\end{proposition}
\begin{proof}
Suppose that $S(W) = S(\bw)$ for some word $\bw \in W$. Since $S(\bw)$ is a nilpotent flat semiring containing a unique annihilator $\bw$, it follows from Lemma~\ref{nilpotentsi} that $S(\bw)$ is subdirectly irreducible.

Conversely, suppose that $S(W)$ is subdirectly irreducible.
If $S(W)$ is not nil, then there exists a word $\bu \in W^{\leq}$ such that $\bu^n \in W^{\leq}$ for all $n\geq 1$.
For each $n\geq 1$, let $I_n$ denote the multiplicative ideal of $S(W)$ generated by $\bu^n$.
Then $I_n = M(W) \bu^n M(W)$, and it is easy to see that the intersection $\bigcap_{n \geq 1} I_n = \{0\}$.
This implies that $S(W)$ has no least nonzero ideal,
contradicting the assumption that $S(W)$ is subdirectly irreducible.
Hence $S(W)$ must be nil.
Now by Lemma~\ref{nilsi}, $S(W)$ contains a unique annihilator $\bw$ and $\bw \in M(W) \bu M(W)$ for all $\bu \in W$.
Consequently, every word in $W$ is a subword of $\bw$. Therefore, $S(W) = S(\bw)$.
\end{proof}


It is easy to see that every flat semiring $S(W)$ satisfies the identity
\begin{equation} \label{xxy}
x + xy \approx 0.
\end{equation}
In fact, the identity~\eqref{xxy} characterizes, to some extent, the property that a flat semiring is nil; see the following lemma.

\begin{lemma}\label{xxynil}
Let $S$ be a flat semiring.
If $S$ is nil, then $S$ satisfies the identity~\eqref{xxy}.
Conversely, if $S$ is finite and satisfies the identity~\eqref{xxy}, then $S$ is nilpotent.
\end{lemma}

\begin{proof}
Suppose that the flat semiring $S$ is nil with zero $0$.
Assume, for contradiction, that $S$ does not satisfy the identity~\eqref{xxy}.
Then there exist $a, b \in S$ such that $a + ab \neq 0$.
Since $S$ is flat, it follows that $a = ab \neq 0$, and so
\[
0 \neq a = ab = ab^2 = \cdots = ab^n = ab^{n+1}=\cdots.
\]
Since $S$ is nil, there exists $n\geq 1$ such that $b^n = 0$.
Thus $a = ab^n = 0$, a contradiction.
Therefore, $S$ satisfies the identity~\eqref{xxy}.

Conversely, suppose that $S$ is a finite flat semiring satisfying the identity~\eqref{xxy}.
Let $a$ be an arbitrary nonzero element of $S$.
Since $S$ is finite, there exist $n, k\geq 1$ such that $a^n = a^{n+k}$, and so
\[
a^n =a^n + a^n=a^n + a^{n+k}=a^n + a^n a^k = 0.
\]
Here the last equality uses the identity~\eqref{xxy}.
Thus $S$ is nil. By Remark~\ref{finnil}, $S$ is nilpotent.
\end{proof}

\begin{proposition}\label{WiW}
Let $\{W_{i}\}_{i\in I}$ be a family of nonempty subsets of $X^+$. Then
\[
\textstyle \mathsf{V}\left(S\left(\bigcup_{i\in I}W_{i}\right)\right) = \mathsf{V}( S(W_{i}) \mid i\in I).
\]
\end{proposition}
\begin{proof}
Let $W$ denote the union $\bigcup_{i\in I}W_{i}$.
For each $i \in I$, define $J_i = S(W) \setminus W_i^{\leq}$.
It is straightforward to verify that $J_i$ is a multiplicative ideal of $S(W)$.
Moreover, the ideal quotient $S(W)/J_i$ is isomorphic to $S(W_i)$, and the intersection $\bigcap_{i \in I} J_i = \{0\}$.
By Proposition~\ref{siflat}, we deduce that $S(W)$ is isomorphic to a subdirect product of the family $(S(W_i))_{i \in I}$.
Consequently,
\[
\mathsf{V}(S(W)) = \mathsf{V}\bigl( S(W_i) \mid i \in I \bigr),
\]
as required.
\end{proof}

\begin{corollary}\label{coro26071601}
Let $W$ and $W'$ be nonempty subsets of $X^+$.
If $W$ is a subset of $W'$, then ${\mathsf V}(S(W))$ is a subvariety of ${\mathsf V}(S(W'))$.
\end{corollary}

\begin{corollary}\label{wwi}
Let $W$ be a nonempty set of words in $X^+$.
Then
\[
{\mathsf V}(S(W))={\mathsf V}(S(\bw) \mid \bw\in W) = {\mathsf V}\left(S(\bw) \mid \bw\in W^{\leq}\right).
\]
In particular, $S(\bw)\in \mathsf{V}(S(W))$ for each $\bw\in W^{\leq}$.
\end{corollary}
\begin{proof}
The equality ${\mathsf V}(S(W))={\mathsf V}(S(\bw) \mid \bw\in W)$ follows directly from Proposition~\ref{WiW}.
The equality ${\mathsf V}(S(W))={\mathsf V}\left(S(\bw) \mid \bw\in W^{\leq}\right)$
follows from Proposition~\ref{WiW} together with the elementary fact that $S(W)=S(W^{\leq})$.
\end{proof}

Two words are \emph{similar}
if one can be obtained from the other by renaming the variables;
that is, they differ only in the names of the variables \cite{carpi2002}.
For instance, $xyx$ and $aba$ are similar.
For any two words $\mathbf{u}$ and $\mathbf{v}$ in $X^+$,
the flat semirings $S(\mathbf{u})$ and $S(\mathbf{v})$ are isomorphic if and only if
$\mathbf{u}$ and $\mathbf{v}$ are similar.

For a nonempty set $W$ of words in $X^+$,
let $\overline{W}$ be a chosen set of representatives of the similarity classes of $W$;
that is, $\overline{W}$ is a subset of $W$, and every word in $W$ is similar to a unique word in $\overline{W}$.
Then $\overline{W}$ is finite if and only if the words in $W$ have bounded length.

\begin{proposition}\label{WW'}
Let $W$ be a nonempty set of words in $X^+$.
Then
\[
\mathsf{V}(S(W)) = \mathsf{V}(S(\overline{W})).
\]
\end{proposition}
\begin{proof}
Since $\overline{W} \subseteq W$, it follows from Corollary~\ref{coro26071601} that
$\mathsf{V}(S(\overline{W}))$ is a subvariety of $\mathsf{V}(S(W))$.
Conversely, by the hypothesis and Corollary~\ref{wwi}, we have
\[
\mathsf{V}(S(W)) = \mathsf{V}\bigl( S(\mathbf{w}) \mid \mathbf{w} \in W \bigr)
\subseteq \mathsf{V}\bigl( S(\mathbf{w}) \mid \mathbf{w} \in \overline{W} \bigr)
= \mathsf{V}(S(\overline{W})).
\]
Thus $\mathsf{V}(S(W))$ is a subvariety of $\mathsf{V}(S(\overline{W}))$,
and so the two varieties are equal.
\end{proof}

Using properties of word similarities, the following proposition characterizes nilpotency of $S(W)$.

\begin{proposition}\label{nilp}
Let $W$ be a nonempty set of words in $X^+$.
Then the following statements are equivalent:
\begin{enumerate}[$(\rm i)$]
\item The flat semiring $S(W)$ is nilpotent.
\item The words in $W$ have bounded length.
\item The variety $\mathsf{V}(S(W))$ is finitely generated.
More precisely, $\mathsf{V}(S(W))$ is generated by the finite flat semiring $S(\overline{W})$.
\end{enumerate}
\end{proposition}

\begin{proof}
$(\rm i)\Rightarrow(\rm ii)$.
Assume that $S(W)$ is nilpotent. Then it is $k$-nilpotent for $k\geq 2$.
This implies that every word in $W$ has length less than $k$.
So the words in $W$ have bounded length.

$(\rm ii)\Rightarrow(\rm iii)$. Suppose that every word in $W$ has length less than $k$ for some $k\geq 1$. Then $\overline W$ is finite.
By Proposition~\ref{WW'}, $\mathsf{V}(S(W))$ is generated by the finite flat semiring $S(\overline{W})$.

$(\rm iii)\Rightarrow(\rm i)$.
Assume that $\mathsf{V}(S(W))$ is finitely generated.
By Lemma~\ref{finsubf}, the variety $\mathsf{V}(S(W))$ is generated by some finite flat semiring $S$,
and so $S(W)$ and $S$ satisfy the same identities.
Since $S(W)$ satisfies identity~\eqref{xxy},
so does $S$. By Lemma~\ref{xxynil},  $S$ is nilpotent,
and therefore satisfies the identity $x_1 \cdots x_k \approx 0$ for some $k\geq 1$.
Consequently, $S(W)$ also satisfies this identity, which means $S(W)$ is nilpotent.
Then, by Proposition~\ref{WW'}, $\mathsf{V}(S(W))$ is generated by the finite flat semiring $S(\overline{W})$.
\end{proof}

For each integer $k\geq 1$, let $W_k$ denote the set of all words in $X^+$ of length $k$.
Then $W_k^{\leq}=\bigcup_{1\leq i\leq k} W_i$,
which consists of all words in $X^+$ of length at most $k$;
the flat semiring $S(W_k)$ is exactly $(k+1)$-nilpotent and exactly $(k+1)$-nil.
Let $\mathcal{W}_k$ denote the variety generated by $S(W_k)$.
By proposition~\ref{nilp}, $\mathcal{W}_k$ is finitely generated.
Since
\[
W_1^{\leq} \subseteq W_2^{\leq} \subseteq \cdots \subseteq W_k^{\leq} \subseteq W_{k+1}^{\leq}\subseteq \cdots,
\]
it follows from Corollary~\ref{wwi} that
\begin{equation}\label{chain071501}
\mathcal{W}_1 \leq \mathcal{W}_2 \leq \cdots \leq
\mathcal{W}_k \leq \mathcal{W}_{k+1} \leq \cdots.
\end{equation}
The identity $x^{k+1} \approx 0$ is satisfied by $S(W_k)$,
but does not hold in $S(W_{k+1})$. So the ascending chain \eqref{chain071501} is strict.
We therefore have the following proposition.

\begin{proposition}
For each $k\geq 1$, the variety $\mathcal{W}_k$ is finitely generated.
Moreover, these varieties form an infinite strictly ascending chain:
\begin{equation}\label{chain}
\mathcal{W}_1 < \mathcal{W}_2 < \cdots < \mathcal{W}_k < \mathcal{W}_{k+1} <\cdots.
\end{equation}
The join of all these varieties equals $\mathsf{V}\bigl(S(X^+)\bigr)$.
\end{proposition}
\begin{proof}
Since $X^+=\bigcup_{k \geq 1} W_k$, Proposition~\ref{WiW} gives
\[
\mathsf{V}(S(X^+)) = \mathsf{V}(S(W_k) \mid k\geq 1),
\]
which is the join of the varieties $\mathsf{V}(S(W_k))$ for $k \geq 1$.
Since $\mathsf{V}(S(W_k))=\mathcal{W}_k$ for every $k\geq 1$,
we obtain that $\mathsf{V}\bigl(S(X^+)\bigr)$ is the join of $(\mathcal{W}_k)_{k\geq 1}$.
\end{proof}

\begin{proposition}\label{knsw}
Let $W$ be a nonempty set of words in $X^+$, and let $k\geq 1$ be an integer.
If every word in $W$ has length at most $k$,
then $\mathsf{V}(S(W))$ is a finitely generated subvariety of $\mathcal{W}_k$.
\end{proposition}
\begin{proof}
Suppose that every word in $W$ has length at most $k$.
By Proposition~\ref{nilp}, $\mathsf{V}(S(W))$ is finitely generated.
Since $W$ is a subset of $W_k^{\leq}$,
it follows from Corollary~\ref{wwi} that $\mathsf{V}(S(W))$ is a subvariety of $\mathcal{W}_k$.
\end{proof}

We shall show in later sections that $\mathcal{W}_k$ is finitely based if and only if $k\leq 3$.
In the next section, we prove the first half of this result by establishing that $\mathcal{W}_3$ is a Cross variety.

Proposition~\ref{WW'} allows us to assume, when studying the variety $\mathsf{V}(S(W))$,
that no two distinct words in $W$ are similar.
This assumption often simplifies the analysis considerably.
With this convention, we may take, for example,
$W_3 = \{a^3, a^2b, ab^2, aba, abc\}$.
Consequently, $\mathcal{W}_3$ is generated by the five flat semirings
$S(a^3), S(a^2b), S(ab^2), S(aba)$, and $S(abc)$.


\section{The finite basis problem for $\mathcal{W}_3$}\label{sec:w3}
In this section,
we show that the flat semiring $S(W)$ is finitely based
if every word in $W$ has length at most $3$ (equivalently, if $S(W)$ is $4$-nilpotent).
To this end, we establish a stronger result: the variety $\mathcal{W}_3$ is a Cross variety.
In doing so, we obtain a complete description of the finite subdirectly irreducible members of $\mathcal{W}_3$.

Recall that an ai-semiring variety is a \emph{Cross variety} if it is
finitely generated, finitely based, and has only finitely many subvarieties.
From \cite[Proposition 1.4.35]{sapir14},
every subvariety of a Cross variety is a Cross variety.
Consequently, every algebra in a Cross variety is finitely based.

We shall also need the following two classes of graph semirings from \cite{gr}.
The \emph{path graph semiring} $S_{\bp_n}$ ($n \geq 2$) over a path $\bp_n = a_1 \to a_2 \to \cdots \to a_n$
is a flat semiring with underlying set $\{0, \omega, a_1, \ldots, a_n\}$,
where multiplication is defined by $a_k a_{k+1} = \omega$ for $1 \leq k \leq n-1$ and all other products are $0$.
The \emph{cycle graph semiring} $S_{\bc_m}$ ($m \geq 1$) over a cycle
$\bc_m = a_1 \to a_2 \to \cdots \to a_m \to a_1$
is a flat semiring with underlying set $\{0, \omega, a_1, \ldots, a_m\}$,
where multiplication is defined by $a_ma_1=\omega$ and $a_k a_{k+1} = \omega$ for $1 \leq k \leq m-1$ and all other products are $0$.
Each $S_{\mathbf{p}_{n}}$ is a subalgebra of both $S_{\mathbf{p}_{n+1}}$ and $S_{\mathbf{c}_{n+1}}$.
Moreover, $S_{\bp_2} \cong S(ab)$, $S_{\bc_1} \cong S(a^2)$, and $S_{\bc_2} \cong S_c(ab)$.

The following lemma provides a complete characterization of finite nontrivial subdirectly irreducible $3$-nilpotent flat semirings in terms of path and cycle graph semirings.
\begin{lemma}[{\cite[Proposition 2.4]{gr}}]\label{graphsemiring}
Every path graph semiring and cycle graph semiring is a subdirectly irreducible $3$-nilpotent flat semiring.
Conversely, every finite nontrivial subdirectly irreducible $3$-nilpotent flat semiring
is isomorphic to a $\{0,\omega\}$-direct union of flat semirings,
each of which is either a path graph semiring or a cycle graph semiring.
\end{lemma}

For ai-semiring varieties $\mathcal{V}_1$ and $\mathcal{V}_2$,
we write $\mathcal{V}_1 < \mathcal{V}_2$ if $\mathcal{V}_1$ is a proper subvariety of $\mathcal{V}_2$.
We shall need the following three technical lemmas concerning $\{0,\omega\}$-direct unions of graph semirings and related flat semirings.

\begin{lemma}[{\cite[Lemma 3.1]{gr}}]\label{scnsln1}
Let $k\geq 1$ and $m,n\geq 2$ be integers. Then
\begin{enumerate}[$(\rm i)$]
\item \label{eq250602011}
$\mathsf{V}(S_{\bp_{n}})<\mathsf{V}(S_{\bp_{n}}\circ S_{\bp_n})
=\mathsf{V}\left(\bigcup_{1\leq i\leq m}^{\omega}S_{\bp_{n}}\right)$;

\item \label{eq250602012}
$\mathsf{V}(S_{\bp_{n}}\circ S_{\bp_n})<\mathsf{V}(S_{\bp_{n+1}})
=\mathsf{V}\left(S_{\bp_{n+1}}\circ \bigcup_{1\leq i\leq m}^{\omega}S_{\bp_{n}}\right)$;

\item\label{eq25060202}
$\mathsf{V}(S_{\bc_1})<\mathsf{V}(S_{\bc_1}\circ S_{\bc_1})
=\mathsf{V}\left(\bigcup_{1\leq i\leq m}^{\omega}S_{\bc_1}\right)$;

\item \label{eq25060203}
$\mathsf{V}\left(\bigcup_{1\leq i\leq m}^{\omega}S_{\bc_n}\right)
=\mathsf{V}(S_{\bc_n})$.
\end{enumerate}
\end{lemma}

\begin{lemma}\label{w3a3aba}
Let $m\geq 1$ be an integer and let $\bw$ be a word of length $3$.
\begin{enumerate}[$(\rm i)$]
\item If $|c(\bw)|\geq 2$, then $S(\bw) \circ \bigcup_{1 \leq i \leq m}^{\omega} S(a_i b_i) \in \mathsf{V}(S(\bw))$.

\item If $|c(\bw)|=1$, then $S(\bw) \circ \bigcup_{1 \leq i \leq m}^{\omega} S_c(a_i b_i) \in \mathsf{V}(S(\bw))$.
\end{enumerate}
\end{lemma}

\begin{proof}
Let $\bw = x_1 x_2 x_3$, where any pair of $x_i$ and $x_j$ may be equal, and denote by $S$ the direct product of $m+1$ copies of $S(x_1 x_2 x_3)$.
Let $A$ be the subalgebra of $S$ generated by the following elements
\[
\mathbf{x}_1 = (x_1, x_1, \ldots, x_1), \quad
\mathbf{x}_2 = (x_2, x_2, \ldots, x_2), \quad
\mathbf{x}_3 = (x_3, x_3, \ldots, x_3),
\]
\[
\mathbf{a}_k = \sigma^k(x_1, x_1 x_2, x_1 x_2, \ldots, x_1 x_2), \quad
\mathbf{b}_k = \sigma^k(x_2 x_3, x_3, x_3, \ldots, x_3),
\]
where $k$ ranges from $1$ to $m$. Take $\sigma\colon S \rightarrow S$ to be the right cyclic shift mapping defined by
\[
\sigma(x_1, x_2, \ldots, x_{m+1}) = (x_{m+1}, x_1, \ldots, x_m).
\]
Set $\mathbf{a} = (x_1 x_2 x_3, \ldots, x_1 x_2 x_3) \in S$ and $J = \{ x \in A \mid \mathbf{a} \notin S^{1} x S^{1} \}$.
Then it is easy to verify that $J$ is an ideal of $A$.
In the ideal quotient $A/J$, for each $\mathbf{x} \in A$, write $\overline{\mathbf{x}} = \mathbf{x} / J$.
It is now a routine matter to verify that  $A / J$ is isomorphic to the $\{0, \omega\}$-direct union
\[
\langle \overline{\mathbf{x}_1}, \overline{\mathbf{x}_2}, \overline{\mathbf{x}_3} \rangle \circ \bigcup\nolimits_{1 \leq k \leq m}^{\omega} \langle \overline{\mathbf{a}_k}, \overline{\mathbf{b}_k} \rangle,
\]
where $\langle \overline{\mathbf{x}_1}, \overline{\mathbf{x}_2}, \overline{\mathbf{x}_3} \rangle$ is isomorphic to $S(x_1x_2x_3)$.
We now consider two cases:

\textbf{Case 1.} $|c(\bw)| \geq 2$. Then $\langle \overline{\mathbf{a}_k}, \overline{\mathbf{b}_k} \rangle$ is isomorphic to $S(a_k b_k)$. This implies that $A / J$ is isomorphic to
\[
S(\bw) \circ \bigcup\nolimits_{1 \leq k \leq m}^{\omega} S(a_k b_k),
\]
which lies in the variety $\mathsf{V}(S(\bw))$. This completes the proof of part (i).

\textbf{Case 2.} $|c(\bw)| = 1$.  Then $\langle \overline{\mathbf{a}_k}, \overline{\mathbf{b}_k} \rangle$ is isomorphic to $S_c(a_k b_k)$. This implies that $A / J$ is isomorphic to
\[
S(\bw) \circ \bigcup\nolimits_{1 \leq k \leq m}^{\omega} S_c(a_k b_k),
\]
which lies in the variety $\mathsf{V}(S(\bw))$. This completes the proof of part (ii).
\end{proof}

\begin{lemma}\label{abap3}
Let $m$ be a positive integer.
\begin{enumerate}[$(\rm i)$]
\item $S(aba) \circ \bigcup_{1 \leq i \leq m}^{\omega} S_{\mathbf{p}_3} \in \mathsf{V}(S(aba) \circ S_{\mathbf{p}_3})$.
\item $S(aba)\circ S_{\mathbf{p}_{3}}\in{\mathsf V}(S(a^{3}),S(aba))$.
\end{enumerate}

\end{lemma}

\begin{proof}
(i) Let $T$ denote the direct product of $m$ copies of $S(aba) \circ S_{\mathbf{p}_3}$. Let $A$ be the subalgebra of $T$ generated by the following elements
\[
(a, a, \ldots, a), \quad (b, b, \ldots, b),
\]
\[
\sigma^k(ab, a_1, a_1, \ldots, a_1), \quad
\sigma^k(a, a_2, a_2, \ldots, a_2), \quad
\sigma^k(ba, a_2, a_3, \ldots, a_3),
\]
where $k$ ranges from $1$ to $m$, and $\sigma\colon T \rightarrow T$ is the right cyclic shift mapping defined by
\[
\sigma(x_1, x_2, \ldots, x_m) = (x_m, x_1, \ldots, x_{m-1}).
\]
Let $J$ be the set of all elements of $A$ that have a $0$-coordinate. Then it is easy to verify that $J$ is an ideal of $A$ and the ideal quotient $A/J$ is isomorphic to $S(aba) \circ \bigcup\nolimits_{1 \leq i \leq m}^{\omega} S_{\mathbf{p}_3}$. Thus, $S(aba) \circ \bigcup\nolimits_{1 \leq i \leq m}^{\omega} S_{\mathbf{p}_3} \in \mathsf{V}(S(aba) \circ S_{\mathbf{p}_3})$.

To prove (ii), it suffices to verify that $S(aba) \circ S_{\mathbf{p}_{3}}$ satisfies every nontrivial identity of the form $\bu \approx \bu + \mathbf{q}$ that is satisfied by both $S(a^{3})$ and $S(aba)$. Let
\[
\varphi \colon c(\bu+\bq) \to S(aba) \circ S_{\mathbf{p}_{3}}
\]
be an arbitrary assignment. We need to show that $\varphi(\bu)=\varphi(\bu+\bq)$.

If $\varphi(\bu)=0$, then $\varphi(\bu)=0=\varphi(\bu+\mathbf{q})$.
If $\varphi(\bu) \neq 0$, then every word in $\bu$ has length less than $4$,
since $S(aba) \circ S_{\mathbf{p}_{3}}$ is $4$-nilpotent.
We may write
\[
\bu = L_{1}(\bu) + L_{2}(\bu) + L_{3}(\bu),
\]
where $L_k(\bu)=\{\mathbf{u_i}\in \bu \mid \ell(\mathbf{u_i})=k\}$ for $k=1,2,3$.

It is easy to see that
$\bu \neq L_{1}(\bu)$, since $\bu \approx \bu + \mathbf{q}$ is nontrivial.
Thus, $\varphi(\bu) \in \{ab, ba, \omega\}$. Consider the following two cases:

\textbf{Case 1.} $\varphi(\bu) \in \{ab, ba\}$. Then $L_{3}(\bu) = \emptyset$, $\varphi(x) \in \{ab, ba\}$ for all $x \in L_{1}(\bu)$, and $(\varphi(x), \varphi(y)) \in \{(a,b), (b,a)\}$ for each $xy \in L_{2}(\bu)$.
Thus, $\varphi$ can be viewed as an assignment from $c(\bu+\bq)$ to $S(aba)$. Since $S(aba)$ satisfies $\bu \approx \bu + \bq$, it follows that $\varphi(\bu) = \varphi(\bu+\bq)$.

\textbf{Case 2.} $\varphi(\bu) = \omega$ (the unique annihilator in $S(aba)\circ S_{\bp_3}$). This implies that
\begin{enumerate}[$(\rm a)$]
    \item $\varphi(x) = \omega$ for all $x \in L_{1}(\bu)$;
    \item $\{\varphi(x), \varphi(y)\} \in \{(ab,a), (a,ba), (a_{1},a_{2}), (a_{2},a_{3})\}$ for all $xy \in L_{2}(\bu)$; and
    \item $(\varphi(x), \varphi(y), \varphi(z)) = (a,b,a)$ for all $xyz \in L_{3}(\bu)$.
\end{enumerate}
Define assignments $\psi_{1}: c(\bu+\bq) \to S(aba)$ and $\psi_{2}: c(\bu+\bq) \to S(a^{3})$ by
\[
\psi_1(x) = \begin{cases}
ab, & \varphi(x) = a_1 \\
a, & \varphi(x) = a_2 \\
ba, & \varphi(x) = a_3 \\
\varphi(x), & \text{otherwise}
\end{cases}, \qquad
\psi_2(x) = \begin{cases}
a^3, & \varphi(x) = \omega \\
a^2, & \varphi(x) \in \{ab, ba, a_2\} \\
a, & \varphi(x) \in \{a, b, a_1, a_3\}
\end{cases}.
\]
Thus $\psi_1(\bu) = aba$ and $\psi_2(\bu) = a^3$. Since the flat semirings $S(aba)$ and $S(a^3)$ satisfy the identity $\bu \approx \bu + \bq$, it follows that $\psi_1(\bq)=aba$ and $\psi_2(\bq)=a^3$; hence $\ell(\bq) \in \{2,3\}$
(the case $\ell(\bq)=1$ gives $\bq \in \bu$, making $\bu \approx \bu + \bq$ trivial).

\textbf{Subcase 2.1.} $\ell(\bq) = 3$, say $\bq = x_1 x_2 x_3$. It is easy to see that
\[
(\psi_{1}(x_{1}), \psi_{1}(x_{2}), \psi_{1}(x_{3})) = (a, b, a).
\]
Thus, $\varphi(x_{2})  = b$ and $\varphi(x_1), \varphi(x_3) \in \{a, a_2\}$.
Also, since $\psi_2(x_1 x_2 x_3) = a^3$, it follows that $\psi_{2}(x_{1}) = \psi_{2}(x_{3}) = a$, and so
$\varphi(x_1), \varphi(x_3) \in \{a, a_1, a_3, b\}$.
Therefore,
\[
\varphi(x_{1}), \varphi(x_{3}) \in \{a, a_{2}\} \cap \{a, b, a_{1}, a_{3}\} = \{a\},
\]
and so $\varphi(x_1) = \varphi(x_3) = a$.
This implies that $\varphi(\bq) = aba = \varphi(\bu)$, and so $\varphi(\bu)=\varphi(\bu+\bq)$, as required.

\textbf{Subcase 2.2.} $\ell(\bq) = 2$, say $\bq = x_1 x_2$. It is easy to see that $(\psi_{1}(x_{1}), \psi_{1}(x_{2})) \in \{(ab,a), (a,ba)\}$ and $(\psi_{2}(x_{1}), \psi_{2}(x_{2})) \in \{(a,a^{2}), (a^{2},a)\}$. Thus,
\[
(\varphi(x_1), \varphi(x_2)) \in \{(ab,a_1) \times \{a,a_2\} \cup \{a,a_2\} \times \{ba,a_3\} \} \triangleq A_1,
\]
\[
(\varphi(x_1), \varphi(x_2)) \in \{(a,b,a_1,a_3) \times \{ab,ba,a_2\} \cup \{ab,ba,a_2\} \times \{a,b,a_1,a_3\}\} \triangleq A_2.
\]
It is now a routine matter to check that
\[
(\varphi(x_1), \varphi(x_2)) \in A_1 \cap A_2 = \{(ab,a), (a,ba), (a_1,a_2), (a_2,a_3)\}.
\]
This implies that $\varphi(\bq) = \varphi(x_{1})\varphi(x_{2}) = \omega = \varphi(\bu)$, and so $\varphi(\bu) = \varphi(\bu+\bq)$.
\end{proof}

With the necessary preparatory results established,
we now prove that the variety $\mathcal{W}_3$ is a Cross variety.
Our strategy is as follows.
Recall that $\mathcal{W}_3$ is generated by the following five flat semirings:
\[
S(a^3),\; S(a^2b),\; S(ab^2),\; S(aba),\; S(abc).
\]
From these five algebras, we extract a finite set of identities that they all satisfy,
and define $\mathcal{V}$ to be the subvariety of $\mathbf F$ determined by these identities.
Obviously, $\mathcal{W}_3$ is a subvariety of $\mathcal{V}$, and $\mathcal{V}$ is finitely based.
We then prove that $\mathcal{V}$ is a Cross variety by showing that it is finitely generated,
and has only finitely many subvarieties.
This is achieved through a complete description of the finite subdirectly irreducible members of
$\mathcal{V}$ and the varieties they generate.
Finally, we prove that $\mathcal{V}$ is a subvariety of $\mathcal{W}_3$,
thereby establishing $\mathcal{V} = \mathcal{W}_3$ and completing the proof.

Let $\mathcal{V}$ be the subvariety of the variety $\mathbf{F}$ defined by the following identities:
\begin{align}
&x_1x_2x_3x_4 \approx y^4, \label{id1}\\
&x_1x_2x_3+y^2 \approx y^4, \label{id11}\\
&x^2+yz \approx y^2+xz, \label{id4}\\
&x^3+yz \approx x^3+zy, \label{id12}\\
&xy+yx+zt \approx xy+yx+tz, \label{id5}\\
&x_1x_2+x_2x_3+x_3x_1 \approx x_1^2+x_2x_3, \label{id3}\\
&x_1x_2+x_2x_3+x_3x_4 \approx x_1x_2+x_2x_3+x_4x_3, \label{id2}\\
&x_1x_2x_3+x_3y \approx x_1x_2x_3+x_3y+x_3x_2x_3, \label{id7}\\
&x_1x_2x_3+yx_1 \approx x_1x_2x_3+yx_1+x_1x_2x_1, \label{id8}\\
&x_1x_2x_3+x_2y \approx x_1x_2x_3+x_2y+x_2x_1x_3, \label{id9}\\
&x_1x_2x_3+yx_2 \approx x_1x_2x_3+yx_2+x_1x_3x_2, \label{id10}\\
&x^2y+x_1x_2+x_2x_3 \approx x^2y+x_2x_1+x_3x_2, \label{id13}\\
&xy^2+x_1x_2+x_2x_3 \approx xy^2+x_2x_1+x_3x_2, \label{id14}\\
&xyz+x_1x_2+x_2x_3 \approx zyx+x_1x_2+x_2x_3, \label{id15}\\
&x_1x_2x_3+y_1y_2y_3 \approx (x_1+y_1)(x_2+y_2)(x_3+y_3). \label{id6}
\end{align}

\begin{proposition}\label{pro26071720}
The variety $\mathcal{V}$ is finitely based.
\end{proposition}
\begin{proof}
By definition, $\mathcal{V}$ is finitely based within $\mathbf{F}$.
Since $\mathbf{F}$ itself is finitely based by Lemma~\ref{lem24121301},
it follows that $\mathcal{V}$ is finitely based.
\end{proof}

An algebra is \emph{locally finite} if each of its finitely generated subalgebras is finite;
a variety is locally finite if all its members are locally finite.
\begin{proposition}\label{vlocal}
The variety $\mathcal{V}$ is locally finite.
\end{proposition}
\begin{proof}
Let $S$ be an arbitrary ai-semiring in $\mathcal{V}$.
Then $S$ satisfies the identity~\eqref{id1}; that is, the product of any four elements in $S$ is zero.
This implies that the multiplicative reduct of $S$ is locally finite.
By \cite[Theorem 2.11(1)]{jrz}, every ai-semiring whose multiplicative reduct is locally finite is itself locally finite.
Applying this to $S$, we conclude that $S$ is locally finite.
\end{proof}

Since every locally finite variety is generated by its finite subdirectly irreducible members,
we now focus on these members in $\mathcal{V}$.
Two ai-semirings are \emph{equationally equivalent} if they generate the same variety.

\begin{proposition}\label{vfsi}
Up to isomorphism, the following algebras are a complete list of the
finite nontrivial subdirectly irreducible members of $\mathcal{V}$:
\begin{itemize}
\item $S(a)$;

\item $S_{\mathbf{c}_{1}}$;

\item $\bigcup_{1\leq i\leq m}^{\omega} S_{\mathbf{c}_{2}}$, $m\geq 1$;

\item $(\bigcup_{1\leq i\leq m}^{\omega} S_{\bp_{2}})\circ  (\bigcup_{1\leq i\leq n}^{\omega} S_{\bp_{3}})$, $m, n\geq 0$, $m+n>0$;

\item $S(a^3)\circ \bigcup_{1\leq i\leq m}^{\omega} S_{c}(a_i b_i)$, $m\geq 0$;

\item $S(a^2b)\circ \bigcup_{1\leq i\leq m}^{\omega} S_c(a_i b_i)$, $m\geq 0$;

\item $S(ab^2)\circ \bigcup_{1\leq i\leq m}^{\omega} S_c(a_i b_i)$, $m\geq 0$;

\item $S(aba)\circ\left(\bigcup_{1\leq i\leq n}^{\omega} S_{\bp_2}\right)\circ \left(\bigcup_{1\leq j\leq m}^{\omega} S_{\mathbf{p}_3}\right)$, $m, n\geq 0$;

\item $S(abc)\circ \bigcup_{1\leq i\leq m}^{\omega} S(a_ib_i)$, $m\geq 0$.
\end{itemize}

Let $\mathcal{C}$ denote the set
\[
\mathcal{C}=\{S(a),\,S(a^{2}),\,S_c(ab),\,S(ab),\,S(ab)\circ S(ab),\,S_{\bp_3},\,S_{\bp_3}\circ S_{\bp_3},
\]
\[
S(a^{3}),\,S(a^{2}b),\,S(ab^{2}),\,S(aba),\,S(aba)\circ S_{\bp_3},\,S(abc)\}.
\]
Then each algebra in $\mathcal{C}$ is a finite nontrivial subdirectly irreducible member of $\mathcal{V}$.
Conversely, every finite nontrivial subdirectly irreducible member of $\mathcal{V}$
is equationally equivalent to some algebra in $\mathcal{C}$.
\end{proposition}

\begin{proof}
Let $\mathcal{V}_{si}$ denote the class of all finite nontrivial subdirectly irreducible algebras in $\mathcal{V}$,
and let $\mathcal{S}$ denote the collection of algebras listed in Proposition~\ref{vfsi}.
The proof proceeds as follows.
We first show that every member of $\mathcal{C}$ lies in $\mathcal{V}_{si}$.
Next, we prove that every algebra in $\mathcal{S}$ is subdirectly irreducible and equationally equivalent to some algebra in $\mathcal{C}$.
This implies that every member of $\mathcal{S}$ lies in $\mathcal{V}_{si}$.
Finally, we show that every algebra in $\mathcal{V}_{si}$ is isomorphic to some algebra in $\mathcal{S}$;
this is the main body of the proof.
Once these three steps are established,
it will follow that, up to isomorphism, $\mathcal{V}_{si}$ coincides with $\mathcal{S}$.

Applying Lemma~~\ref{0wunion}, Proposition~\ref{swsi} and Lemma~\ref{graphsemiring},
one verifies that every algebra in $\mathcal{C}$ is subdirectly irreducible.
A routine check shows that every algebra in $\mathcal{C}$ satisfies \eqref{id1}--\eqref{id15}.
Thus every member of $\mathcal{C}$ lies in $\mathcal{V}_{si}$.

The same propositions also yield that
every algebra in $\mathcal{S}$ is subdirectly irreducible.
It remains to show that each of them is equationally equivalent to some algebra in $\mathcal{C}$.
This is immediate for $S(a)$ and $S_{\bc_1}$, the latter being isomorphic to $S(a^2)$.

For algebras of the form $\bigcup_{1\leq i\leq m}^{\omega} S_{\bc_2}$ with $m\geq 1$,
Lemma~\ref{scnsln1} gives $\mathsf{V}(S)=\mathsf{V}(S_c(ab))$.
Similarly, for algebras of the form
$\left(\bigcup_{1\leq i\leq m}^{\omega} S_{\bp_2}\right)\circ \left(\bigcup_{1\leq j\leq n}^{\omega} S_{\bp_3}\right)$
with $m,n\geq 0$ and $m+n>0$,
the same lemma yields that $\mathsf{V}(S)$ is one of
$\mathsf{V}(S(ab))$,
$\mathsf{V}(S(ab)\circ S(ab))$,
$\mathsf{V}(S_{\bp_3})$, or
$\mathsf{V}(S_{\bp_3}\circ S_{\bp_3})$.

For the remaining five types, Lemma~\ref{w3a3aba} gives the following reductions:
\[
 \begin{aligned}
\textstyle\mathsf{V}\left(S(a^3)\circ \bigcup_{1\leq i\leq m}^{\omega} S_c(a_i b_i)\right) &= \mathsf{V}\left(S(a^3)\right),\\
\textstyle\mathsf{V}\left(S(a^2b)\circ \bigcup_{1\leq i\leq m}^{\omega} S_c(a_i b_i)\right) &= \mathsf{V}\left(S(a^2b)\right),\\
\textstyle\mathsf{V}\left(S(ab^2)\circ \bigcup_{1\leq i\leq m}^{\omega} S_c(a_i b_i) \right)&= \mathsf{V}\left(S(ab^2)\right),\\
\textstyle\mathsf{V}\left(S(aba)\circ\left(\bigcup_{1\leq i\leq m}^{\omega} S_{\bp_2}\right)\circ \left(\bigcup_{1\leq j\leq n}^{\omega} S_{\bp_3}\right)\right) &= \mathsf{V}\left(S(aba)\circ S_{\bp_3}\right),\\
\textstyle\mathsf{V}\left(S(abc)\circ \bigcup_{1\leq i\leq m}^{\omega} S(ab)\right) &= \mathsf{V}\left(S(abc)\right).
\end{aligned}
\]
Thus every algebra in $\mathcal{S}$ is equationally equivalent to some algebra in $\mathcal{C}$.

Now
let $S$ be an arbitrary algebra in $\mathcal{V}_{si}$.
It remains to show that $S$ is isomorphic to some algebra in $\mathcal{S}$.
By Lemma~\ref{lem24121301} and Lemma~\ref{nilpotentsi}, $S$ is a flat semiring,
contains a unique annihilator $\omega$ and
$\{0,\omega\}$ is its least nonzero multiplicative ideal.
Since $S$ satisfies the identity~\eqref{id1}, it follows that $S$ is $4$-nilpotent; that is, $S^4=\{0\}$.
This leads naturally to the following three cases:

\textbf{Case 1.} $S^{2}=\{0\}$.
By Lemma~\ref{sk0w}, $S = \{0, \omega\}$, which is isomorphic to the algebra $S(a)$ in $\mathcal{S}$.

\textbf{Case 2.} $S^{2}\neq\{0\}$ and $S^{3}=\{0\}$.
Then $S$ is $3$-nilpotent, but is not $2$-nilpotent.
We shall show that
$S$ is isomorphic to one of the following algebras in $\mathcal{S}$:
\[
S_{\bc_1},\quad
\bigcup\nolimits_{1\leq i\leq m}^{\omega} S_{\bc_2} \ (m\geq 1),
\]
\[
\left(\bigcup\nolimits_{1\leq i\leq m}^{\omega} S_{\bp_2}\right)\circ \left(\bigcup\nolimits_{1\leq j\leq n}^{\omega} S_{\bp_3}\right)
\ (m,n\geq 0,\ m+n>0).
\]

Indeed, by Lemma~\ref{graphsemiring},
$S$ is isomorphic to the $\{0,\omega\}$-direct union of some flat semirings,
each of which is either a path graph semiring $S_{\mathbf{p}_{m}}$ or a cycle graph semiring $S_{\mathbf{c}_{n}}$.

We first exclude all but finitely many types of graph semirings from consideration.
It is easy to see that neither $S_{\mathbf{p}_{4}}$ nor $S_{\mathbf{c}_{4}}$ satisfy the identity~\eqref{id2},
and $S_{\mathbf{c}_{3}}$ does not satisfy the identity~\eqref{id3}.
Moreover, for every $n\geq 2$, the semiring $S_{\mathbf{p}_{n}}$ is a subalgebra of both $S_{\mathbf{p}_{n+1}}$ and $S_{\mathbf{c}_{n+1}}$.
Consequently, for all $m\geq 4$ and $n\geq 3$, neither $S_{\mathbf{p}_{m}}$ nor $S_{\mathbf{c}_{n}}$ can occur as a subsemiring of $S$.
Therefore, $S$ must be a $\{0,\omega\}$-direct union of semirings in the finite set
$\{S_{\mathbf{c}_{1}}, S_{\mathbf{c}_{2}}, S_{\mathbf{p}_{2}}, S_{\mathbf{p}_{3}}\}$.

Next, we further restrict the possible combinations of these components.
A direct verification shows that $S_{\bc_1}\circ S_{\bc_1}$, $S_{\bc_1}\circ S_{\bc_2}$, $S_{\bc_1}\circ S_{\bp_2}$,
and $S_{\bc_1}\circ S_{\bp_3}$ fail to satisfy \eqref{id4},
and that $S_{\bc_2}\circ S_{\bp_2}$ and $S_{\bc_2}\circ S_{\bp_3}$ fail to satisfy \eqref{id5}.
Thus, the only combinations that remain possible are those in which $S$ is
a single $S_{\mathbf{c}_{1}}$ component,
a $\{0,\omega\}$-direct union of $S_{\mathbf{c}_{2}}$ components,
or a $\{0,\omega\}$-direct union of components from $\{S_{\bp_{2}},S_{\bp_{3}}\}$.
This completes the analysis of Case 2.

\textbf{Case 3.} $S^{3}\neq\{0\}$ and $S^{4}=\{0\}$.
By Lemma~\ref{sk0w}, $S^{3}=\{0,\omega\}$.
Since $S$ satisfies the identity~\eqref{id6},
$\omega$ has a unique factorization into a product of three elements;
that is, there exist unique $a, b, c\in S$ such that
\[
\omega=abc.
\]
Note that these three elements are not necessarily distinct.
We now consider the following five subcases.
The first subcase is treated in full detail; the remaining four are analogous,
and we shall therefore omit the repetitive details, indicating only the differences.

\textbf{Subcase 3.1.} $a=b=c$. Then $\omega=a^3$.
We first claim that
\begin{equation}\label{aa=aa}
\{(x,y)\in S\times S \mid xy\in S\setminus\{0,\omega\}\}=\{(a,a)\}.
\end{equation}
Indeed, it is easy to see that $a^2 \notin \{0, \omega\}$.
Conversely, let $x, y\in S$ be such that $xy \in S\setminus\{0, \omega\}$.
Since $\omega \in S^1(xy)S^1$ and $S$ is $4$-nilpotent,
there exists $s\in S$ such that either $s\cdot xy=\omega$ or $xy\cdot s=\omega$.
By the uniqueness of decomposition of $\omega$, it follows immediately that $x = y = s = a$.
Hence \eqref{aa=aa} holds, and so $S^2=\{0,\omega,a^2\}$.

Let $\langle a\rangle$ denote the subsemiring of $S$ generated by $a$.
Then
\[
\langle a\rangle = \{0, \omega, a^2, a\}.
\]
If $S=\langle a\rangle$, then $S$ is isomorphic to the algebra $S(a^3)$ in $\mathcal{S}$.
Suppose now that $S\neq \langle a\rangle$.
Let $T$ denote the set $(S\setminus\langle a\rangle)\cup\{0,\omega\}$.
Then $T$ properly contains $\{0,\omega\}$ and coincides with $S\setminus\{a, a^2\}$. It follows that
$S=\langle a\rangle\cup T$ and $\langle a\rangle\cap T=\{0,\omega\}$.

We first show that $T$ is a subsemiring of $S$.
It is obvious that $T^2\subseteq S^2=\{0,\omega,a^2\}$.
By \eqref{aa=aa}, we have $a^2\notin T^2$, since $a\notin T$.
Consequently, $T^2\subseteq\{0,\omega\}\subseteq T$.
Since $S$ is a flat semiring, it now follows that
$T$ is a subsemiring of $S$.

Next, we prove that $\langle a\rangle\cdot T=\{0\}$.
It is obvious that $\langle a\rangle\cdot T\subseteq S^2=\{0,\omega,a^2\}$.
Since $a\notin T$, it follows from \eqref{aa=aa} that $a^2\notin \langle a\rangle\cdot T$.
If $\omega\in \langle a\rangle\cdot T$, then there exist $x\in\{a,a^2\}$ and $t\in T$ such that $xt=\omega=a^3$.
As $S$ is $0$-cancellative, this forces $t\in\{a,a^2\}$, contradicting $T\cap\langle a\rangle=\{0,\omega\}$.
Hence $\langle a\rangle\cdot T=\{0\}$.
By a symmetric argument, we also obtain $T\cdot\langle a\rangle=\{0\}$.

We have shown that $S$ is isomorphic to the $\{0,\omega\}$-direct union $S(a^3) \circ T$.
It remains to determine the structure of $T$ more precisely.

For any $t\in T\backslash \{0,\omega\}$, $\omega\in S^1tS^1$.
Since $a\notin T$, it follows from the uniqueness of decomposition of $\omega$
that there is $s\in S$ such that $st=\omega$ or $ts=\omega$.
Since $\langle a\rangle\cdot T = T \cdot  \langle a\rangle = \{0\}$,
it follows that $s\in S\backslash \{a,a^2\}=T$,
and so $\omega\in T^1tT^1$.
Consequently, $\omega$ is the unique annihilator of $T$.
Since $T^3\subseteq S^3=\{0, \omega\}$ and $\omega\notin T^3$,
it follows that $T^3=\{0\}$, and so $T$ is $3$-nilpotent.
Therefore, $T$ is subdirectly irreducible.
Notice also that $\{0,\omega\}$ is properly contained in $T$. By Lemma~\ref{sk0w}, $T^2\neq \{0\}$.
Thus $T$ satisfies precisely the hypotheses of Case 2, and so $T$ is isomorphic to one of the three types of algebras described there.

It remains to rule out the possibilities for $T$.
Since $S(a^3)\circ S(a^2)$ fails the identity~\eqref{id11} and
$S(a^3)\circ S(ab)$ fails the identity~\eqref{id12}, Case 2 forces
$T$ to be isomorphic to $\bigcup_{1\leq i\leq m}^{\omega} S_{c}(ab)$
for some $m \geq 1$.
Consequently, $S$ is isomorphic to the algebra
$S(a^3)\circ \bigcup_{1\leq i\leq m}^{\omega} S_{c}(a_i b_i)$
in $\mathcal{S}$.

\textbf{Subcase 3.2.} $a=b\neq c$. Then $\omega=a^2c$.
By an argument analogous to~\eqref{aa=aa} in Subcase 3.1, we obtain
\[
\{(x,y)\in S\times S \mid xy\in S\setminus\{0,\omega\}\}=\{(a,a),(a,c)\}.
\]
Consequently, $S^2=\{0, \omega, a^2, ac\}$.

Let $\langle a, c\rangle$ denote the subsemiring of $S$ generated by $a$ and $c$.
Then
\[
\langle a,c\rangle=\{0,\omega,a,c,a^2,ac\}.
\]
If $S=\langle a,c\rangle$, then $S$ is isomorphic to the algebra $S(a^2c)$ in $\mathcal{S}$.
Suppose now that $S\neq \langle a,c\rangle$.
Let $T$ denote the set $(S\setminus\langle a,c\rangle)\cup\{0,\omega\}$.
Then $T$ properly contains $\{0,\omega\}$, and $T=S\backslash \{a,c,a^2,ac\}$.
It follows that $S=\langle a,c\rangle\cup T$ and $\langle a,c\rangle\cap T=\{0,\omega\}$.
Moreover, one can easily verify that $T$ is a subsemiring of $S$.

We next prove that $\langle a,c\rangle\cdot T=\{0\}$.
Clearly, $\langle a,c\rangle\cdot T\subseteq S^2=\{0,\omega,a^2,ac\}$.
Since $a,c\notin T$, we have $a^2,ac\notin \langle a,c\rangle\cdot T$.
Suppose, for contradiction, that $\omega\in \langle a,c\rangle\cdot T$.
Then there exist $s\in\{a,c,a^2,ac\}$ and $t\in T$ such that $st=\omega$.
Because $S$ is $0$-cancellative and the uniqueness of decomposition of $\omega$,
we must have $s\notin\{a,a^2,ac\}$; hence $s=c$.
Thus $ct=\omega$, and consequently $a^2c+ct=\omega\neq 0$.
Since $S$ satisfies identity~\eqref{id7}, we obtain
\[
a^2c+ct+cac=\omega.
\]
It follows that $cac = \omega$, contradicting the uniqueness of decomposition for $\omega$.
Therefore $\langle a,c\rangle\cdot T=\{0\}$.
By a symmetric argument using identity~\eqref{id8}, we also obtain $T\cdot\langle a,c\rangle=\{0\}$.

We have shown that $S$ is isomorphic to the $\{0,\omega\}$-direct union $S(a^2c) \circ T$.
It remains to determine the structure of $T$ more precisely.

One can show that $T$ is a finite $3$-nilpotent flat semiring containing the unique annihilator $\omega$.
Consequently, $T$ is a finite subdirectly irreducible member of $\mathcal{V}$.
Notice that $T^2\neq \{0\}$ and $T^3= \{0\}$.
Thus $T$ satisfies precisely the hypotheses of Case 2, and so $T$ is isomorphic to one of the three types of algebras described there.

The algebras $S(a^2b)\circ S(c^2)$, $S(a^2b)\circ S_c(cd)$, and $S(a^2b)\circ S_{\bp_3}$
fail \eqref{id11}, \eqref{id5}, and \eqref{id13}, respectively.
Hence Case~2 forces $T$ to be isomorphic to $\bigcup_{1\leq i\leq m}^{\omega} S_c(a_i b_i)$
for some $m\geq 1$.
Consequently,
$S$ is isomorphic to the algebra $S(a^2b)\circ \bigcup_{1\leq i\leq m} S_c(a_i b_i)$ in $\mathcal{S}$.

\textbf{Subcase 3.3.} $b=c\neq a$. Then $\omega=ac^2$.
Note that the flat semirings $S(a^2c)$ and $S(ac^2)$ are anti-isomorphic.
By an argument dual to that used in Subcase 3.2, with the identity~\eqref{id13} replaced by its dual identity~\eqref{id14},
we obtain that $S$ is isomorphic to the algebra $S(ab^2)\circ \bigcup_{1\leq i\leq m}^{\omega} S_c(a_i b_i)$ for some $m\geq 0$.

\textbf{Subcase 3.4.} $a=c\neq b$. Then $\omega=aba$.
One can obtain
\[
\{(x,y)\in S\times S \mid xy\in S\setminus\{0,\omega\}\}=\{(a,a),(a,b),(b,a)\}.
\]
Consequently, $S^2=\{0,\omega,ab,ba\}$.

Let $\langle a, b\rangle$ denote the subsemiring of $S$ generated by $a$ and $b$.
Then
\[
\langle a,b\rangle=\{0,\omega,a,b,ab,ba\}.
\]
If $S=\langle a,b\rangle$, then $S$ is isomorphic to the algebra $S(aba)$ in $\mathcal{S}$.
Suppose now that $S\neq \langle a,b\rangle$.
Let $T$ denote the set $(S\setminus\langle a,b\rangle)\cup\{0,\omega\}$.
Then $T$ properly contains $\{0,\omega\}$.
Using the identities~\eqref{id9} and~\eqref{id10},
we obtain that $S$ is isomorphic to the $\{0,\omega\}$-direct union of $S(aba)$ and $T$.
Moreover, $T$ satisfies precisely the hypotheses of Case 2,
and so $T$ is isomorphic to one of the three types of algebras described there.

The algebras $S(aba)\circ S(a^2)$ and $S(aba)\circ S_c(ab)$ fail \eqref{id11} and \eqref{id5}, respectively.
Thus Case 2 forces $T$ to be isomorphic either to $\bigcup_{1\leq i\leq m}^{\omega} S_{\bp_2}$ for some $m \geq 1$,
or to $\left(\bigcup_{1\leq i\leq n}^{\omega} S_{\bp_2}\right)\circ \left(\bigcup_{1\leq j\leq m}^{\omega} S_{\bp_3}\right)$
for some $n \geq 0$ and $m \geq 1$.
Consequently,
$S$ is isomorphic to $S(aba)\circ \bigcup_{1\leq i\leq m}^{\omega} S(a_i b_i)$
or $S(aba) \circ \left(\bigcup_{1\leq i\leq n}^{\omega} S_{\bp_2}\right)\circ \left(\bigcup_{1\leq j\leq m}^{\omega} S_{\mathbf{p}_3}\right)$,
both of which belong to $\mathcal{S}$.

\textbf{Subcase 3.5.} $a, b, c$ are distinct. Then $\omega=abc$.
Let $\langle a,b,c\rangle$ denote the subsemiring of $S$ generated by $\{a, b, c\}$.
If $S=\langle a,b,c\rangle$, then $S$ is isomorphic to the algebra $S(abc)$ in $\mathcal{S}$.
Suppose now that $S\neq \langle a,b,c\rangle$.
Let $T$ denote the set $(S\setminus\langle a,b,c\rangle)\cup\{0,\omega\}$.
Then $T$ properly contains $\{0,\omega\}$.
It is readily verified that $T$ is a subsemiring of $S$.
Furthermore, using identities~\eqref{id7}--\eqref{id10},
we obtain that $S$ is the $\{0,\omega\}$-direct union $\langle a,b,c\rangle\circ T$.
Moreover, $T$ satisfies precisely the hypotheses of Case 2,
and so $T$ is isomorphic to one of the three types of algebras described there.

The algebras $S(abc)\circ S(d^2)$, $S(abc)\circ S_c(cd)$, and $S(abc)\circ S_{\bp_3}$
fail \eqref{id11}, \eqref{id5}, and \eqref{id15}, respectively.
Hence Case 2 forces $T$ to be isomorphic to $\bigcup_{1\leq i\leq m}^{\omega} S_{\bp_2}$ for some $m \geq 1$.
Consequently, $S$ is isomorphic to the algebra
$S(abc)\circ \bigcup_{1\leq i\leq m}^{\omega} S(a_i b_i)$
in $\mathcal{S}$.

This completes the proof that every algebra in $\mathcal{V}_{si}$ is isomorphic to some algebra in $\mathcal{S}$.
Together with the previous steps, this establishes that, up to isomorphism, $\mathcal{V}_{si} = \mathcal{S}$.
\end{proof}

\begin{corollary}\label{coro26071801}
The variety $\mathcal{V}$ is finitely generated.
More precisely, $\mathcal{V}$ is generated by $\mathcal{C}$.
\end{corollary}
\begin{proof}
By Proposition~\ref{vlocal}, $\mathcal{V}$ is locally finite and hence generated by its finite subdirectly irreducible members.
By Proposition~\ref{vfsi}, each of these generates the same variety as some algebra in the set $\mathcal{C}$,
and every algebra in $\mathcal{C}$ belongs to $\mathcal{V}$.
Consequently, $\mathcal{V}$ is generated by $\mathcal{C}$, which consists of finitely many finite algebras.
Therefore, $\mathcal{V}$ is finitely generated.
\end{proof}

\begin{corollary}\label{coro26071750}
The variety $\mathcal{V}$ has finitely many subvarieties.
\end{corollary}
\begin{proof}
By Proposition~\ref{vlocal}, $\mathcal{V}$ is locally finite.
Let $\mathcal{W}$ be an arbitrary subvariety of $\mathcal{V}$.
Then $\mathcal{W}$ is also locally finite, and hence is generated by its finite subdirectly irreducible members.
Each such member of $\mathcal{W}$ is also a finite subdirectly irreducible member of $\mathcal{V}$.
By Proposition~\ref{vfsi}, each of these generates the same variety as some algebra in $\mathcal{C}$.
Therefore, $\mathcal{W}$ is generated by a subset of $\mathcal{C}$.
Since $\mathcal{C}$ contains $13$ algebras,
$\mathcal{V}$ has at most $2^{13}$ subvarieties.
\end{proof}

\begin{proposition}\label{coro26072850}
The variety $\mathcal{V}$ is a Cross variety.
\end{proposition}
\begin{proof}
This follows immediately from Proposition~\ref{pro26071720} together with Corollaries~\ref{coro26071801} and~\ref{coro26071750}.
\end{proof}

We now arrive at the main theorem of this section.
\begin{theorem}\label{thm26071601}
The variety $\mathcal{W}_3$ is a Cross variety.
\end{theorem}
\begin{proof}
By Proposition~\ref{coro26072850},
it remains to prove that $\mathcal{W}_3 = \mathcal{V}$.
Indeed, it is easy to check that every flat semiring in the set
\[
\{S(a^3),\; S(a^2b),\; S(ab^2),\; S(aba),\; S(abc)\}
\]
satisfies the identities \eqref{id1}--\eqref{id15},
which define $\mathcal{V}$ as a subvariety of $\mathbf{F}$.
Since $\mathcal{W}_3$ can be generated by these five flat semirings,
it follows immediately that $\mathcal{W}_3$ is a subvariety of $\mathcal{V}$.
For the reverse inclusion, by Corollary~\ref{coro26071801},
it suffices to show that every algebra in $\mathcal{C}$ lies in $\mathcal{W}_3$.

Examining the members of $\mathcal{C}$, we see that every member of $\mathcal{C}$ either is isomorphic to $S(\mathbf{w})$ for some $\mathbf{w}$ with $\ell(\mathbf{w})\leq 3$, or lies in $\mathsf{V}(S(aba)\circ S_{\mathbf{p}_3})$.
The former are clearly in $\mathcal{W}_3$.
For the latter, by Lemma~\ref{abap3},
we have $S(aba)\circ S_{\mathbf{p}_3}$ belong to $\mathsf{V}(S(a^3),S(aba))$,
which is a subvariety of $\mathcal{W}_3$.
Hence every algebra in $\mathcal{C}$ lies in $\mathcal{W}_3$.
Since $\mathcal{V}$ is generated by $\mathcal{C}$, $\mathcal{V}$ is a subvariety of $\mathcal{W}_3$.
Therefore, $\mathcal{W}_3=\mathcal{V}$, and consequently $\mathcal{W}_3$ is a Cross variety.
\end{proof}

\begin{corollary}\label{4swfb}
Let $W$ be a nonempty set of words in $X^+$.
If every word in $W$ has length at most $3$,
then the variety $\mathsf{V}(S(W))$ is a Cross variety.
In particular, $S(W)$ is finitely based.
\end{corollary}
\begin{proof}
Suppose that every word in $W$ has length at most $3$.
By Proposition~\ref{knsw}, $\mathsf{V}(S(W))$ is a subvariety of $\mathcal{W}_3$.
Since every subvariety of a Cross variety is again a Cross variety,
Theorem~\ref{thm26071601} implies that $\mathsf{V}(S(W))$ is a Cross variety.
Therefore, $S(W)$ is finitely based.
\end{proof}

\begin{corollary}\label{coro2026072610}
Let $1\leq k \leq 3$ be an integer. Then the variety $\mathcal{W}_k$ is finitely based.
\end{corollary}
\begin{proof}
This is a direct consequence of Corollary~\ref{4swfb}.
\end{proof}

\begin{remark}
We note in passing that $\mathcal{W}_3$ can be generated by $S(a^3)$, $S(a^2b)$, $S(ab^2)$, and $S(aba)$;
indeed, $S(abc) \in \mathsf{V}(S(a^2b), S(ab^2))$.
We do not include the proof here, as it is not needed for the main results of this paper.
\end{remark}

\section{A sufficient condition for the nonfinitely based property}\label{sec:NFB}
In this section, we use the hypergraph semiring approach to establish a sufficient condition under which an ai-semiring variety is nonfinitely based.
As an application, we show that certain flat semirings of the form $S(W)$ are nonfinitely based.

Following Jackson et al.~\cite{jrz}, we recall the notion of $k$-uniform hypergraphs. 
Let $k \geq 3$ be an integer.
A \emph{$k$-uniform hypergraph} $\mathbb{H}$ is a pair $\langle V, E\rangle$, where $E$ is a family of $k$-element subsets of a set $V$. Each element of $V$ is a \emph{vertex} of $\mathbb{H}$, and each element of $E$ is a \emph{hyperedge} of $\mathbb{H}$.

Let $\mathbb{H}$ be a $k$-uniform hypergraph, and let $n\geq 2$ be a positive integer. Then $\mathbb{H}$ is \emph{$n$-colourable} if there exists a mapping $\varphi: V \to \{1,2,\ldots,n\}$ such that $|\varphi(e)| \geq 2$ for all $e \in E$ (i.e., no hyperedge is monochromatic).
A \emph{cycle} of $\mathbb{H}$ is an alternating sequence $v_1, e_1, v_2, e_2, \ldots, v_n, e_n$ of
distinct vertices and hyperedges such that $v_1 \in e_1 \cap e_n$ and $v_{i+1} \in e_i \cap e_{i+1}$ for $1 \leq i < n$.
The length of this cycle is $n$.
The \emph{girth} of $\mathbb{H}$, denoted by $g(\mathbb{H})$, is the length of its shortest cycles.
If $\mathbb{H}$ contains no cycles, then $\mathbb{H}$ is called a \emph{hyperforest}, and we set $g(\mathbb{H}) = \infty$.

The following result is due to Erd\H{o}s and Hajnal~\cite{eh}; see Theorems 2.6 and 2.7 of~\cite{hj} for further discussion.

\begin{lemma}\label{kml}
For any integers $k,m,\ell \geq 2$, there exists a $k$-uniform hypergraph $\mathbb{H}$ such that $g(\mathbb{H}) \geq \ell$ and $\mathbb{H}$ is not $m$-colourable.
\end{lemma}

Henceforth in this section, every hypergraph $\mathbb{H}$ under consideration is $k$-uniform for some $k \geq 3$,
has no isolated vertices (i.e., each vertex belongs to at least one hyperedge), and satisfies $g(\mathbb{H}) \geq 4$.

Notice that the condition $g(\mathbb{H}) \geq 4$ ensures that any two distinct hyperedges of $\mathbb{H}$ intersect in at most one vertex,
and that for all $1<\ell \leq k$, a set $\{v_1,\ldots,v_{\ell}\}$ is a subhyperedge (that is, a subset of some hyperedge) if and only if every $2$-element subset of $\{v_1,\ldots,v_{\ell}\}$ is a subhyperedge (see~\cite[Lemma 3.2]{jrz}).

We now introduce two concepts that are intimately connected with a hypergraph $\mathbb{H}$:
the hypergraph semiring $S_{\mathbb{H}}$ and the associated polynomial $\bt_{\mathbb{H}}$.

A \emph{hypergraph semiring} $S_{\mathbb{H}}$ defined by $\mathbb{H}$ is a flat semiring generated by a copy $\{\mathbf{a}_v \mid v \in V\}$ of $V$, along with a special element $0$, and subject to the following rules:
\begin{itemize}
\item[(1)] $0$ is the multiplicative zero element;
\item[(2)] $\mathbf{a}_u \mathbf{a}_v = \mathbf{a}_v \mathbf{a}_u$ for all $u, v \in V$;
\item[(3)] $\mathbf{a}_{u_1} \mathbf{a}_{u_2} \cdots \mathbf{a}_{u_k} = \mathbf{a}_{v_1} \mathbf{a}_{v_2} \cdots \mathbf{a}_{v_k}$ whenever $\{u_1, \ldots, u_k\}, \{v_1, \ldots, v_k\} \in E$;
\item[(4)] $\mathbf{a}_{u_1} \mathbf{a}_{u_2} \cdots \mathbf{a}_{u_{k-1}} = \mathbf{a}_{v_1} \mathbf{a}_{v_2} \cdots \mathbf{a}_{v_{k-1}}$ if there exists $v \in V$ such that
    \[\{u_1, \ldots, u_{k-1}, v\}, \{v_1, \ldots, v_{k-1}, v\} \in E.\]
\end{itemize}
We let $\ba$ denote the common value of the products in item (3).
From~\cite[Lemma 3.4]{jrz}, we know that $S_{\mathbb{H}}$ is an exactly $(k+1)$-nilpotent flat semiring with the unique annihilator $\ba$.

The polynomial $\bt_{\mathbb H}$ associated with $\mathbb{H}$ is defined as follows.
Let $\{x_v \mid v \in V\}$ be a set of variables in bijection with $V$, and set
\[
\bt_{\mathbb H} = \sum_{\{v_1,v_2,\ldots,v_k\} \in E} x_{v_1} x_{v_2} \cdots x_{v_k}.
\]
A product \(x_{v_1} x_{v_2} \cdots x_{v_k}\) with \(\{v_1, v_2, \ldots, v_k\} \in E\) will be called a \emph{hyperedge product};
since a hyperedge is an unordered $k$-element subset, it gives rise to $k!$ distinct hyperedge products.
A word over $\{x_v \mid v \in V\}$ is called a \emph{non-hyperedge word}
if its evaluation under the mapping $x_v \mapsto \mathbf{a}_v$ is not $\ba$.

The following result is a modification of \cite[Theorem 2.2]{gjrz}, adapted to the present setting. The original theorem has found applications in a number of contexts (see~\cite{aj,gjrz2}), and the variant we give here provides a sufficient condition for the property of being nonfinitely based for an ai-semiring variety.

To state this result, we first recall the required hypergraph-theoretic setup.
For each $k\geq 3$, we fix a family of $k$-uniform hypergraphs $(\mathbb{H}_n)_{n \geq 1}$,
where $\mathbb{H}_n = \langle V(\mathbb{H}_n), E(\mathbb{H}_n) \rangle$,
such that for each $n \geq 1$,
$\mathbb{H}_n$ is not $2$-colourable and has girth greater than $\binom{kn}{2}$.
The existence of such hypergraphs is guaranteed by Lemma~\ref{kml}.

\begin{proposition}\label{thnfb}
Let $\mathcal{V}$ be an ai-semiring variety that contains $S_c(a_1 a_2 \cdots a_k)$ for some $k\geq 3$,
and for $n \geq 2$ let $\mathbf{w}_n$ be a non-hyperedge word for $\mathbb{H}_n$. If $\mathcal{V}$ satisfies the identities
\begin{equation}\label{eqwn}
\bt_{\mathbb H_n} \approx \bt_{\mathbb H_n} + \mathbf{w}_n
\end{equation}
for all $n \geq 2$, then $\mathcal{V}$ has no finite basis for its equational theory.
\end{proposition}
\begin{proof}
The proof proceeds as follows. We show that for every $n \geq 2$,
the $n$-variable identities of $\mathcal{V}$ do not form a basis for the equational theory of $\mathcal{V}$. To establish this, it suffices to prove that for every $n \geq 2$, the hypergraph semiring $S_{\mathbb{H}_n}$ does not lie in $\mathcal{V}$, while every $n$-generated subalgebra of $S_{\mathbb{H}_n}$ does lie in $\mathcal{V}$.

Let $n \geq 2$ be an integer.
From the proof of~\cite[Theorem 4.9]{jrz}, we know that each $n$-generated subalgebra $T$ of $S_{\mathbb{H}_n}$ is a subalgebra of a hypergraph semiring $S_{\mathbb{G}^+}$, where $\mathbb{G}^+$ is a $k$-uniform subhypergraph of $\mathbb{H}_n$ with at most $\binom{kn}{2}$ hyperedges. Thus $S_{\mathbb{G}^+}$ is a hyperforest semiring. It follows from~\cite[Lemma 4.2]{jrz} that $S_{\mathbb{G}^+} \in \mathsf{V}(S_c(a_1 \cdots a_k))$ (this improves the parameter $k\binom{kn}{2}$ in the proof of~\cite[Theorem 4.9]{jrz} to $\binom{kn}{2}$). Therefore, $T$ belongs to $\mathcal{V}$.

Now consider the assignment $\varphi: \{x_v\mid v\in V(\mathbb H_n)\} \to S_{\mathbb{H}_n}$
defined by $\varphi(x_v) = \mathbf{a}_v$ for all $v \in V(\mathbb H_n)$.
Then $\varphi(\bt_{\mathbb H_n}) = \mathbf{a}$, and $\varphi(\mathbf{w}_n) \neq \mathbf{a}$,
since $\mathbf{w}_n$ is a non-hyperedge word for $\mathbb{H}_n$.
This shows that $S_{\mathbb{H}_n}$ does not satisfy the identity~\eqref{eqwn}.
Hence $S_{\mathbb{H}_n}$ is not a member of $\mathcal{V}$.
Therefore, $\mathcal{V}$ is nonfinitely based as required.
\end{proof}

\begin{remark}\label{remark:5.3}
In the following, we will frequently use the identities obtained from \eqref{eqwn} by taking $\mathbf{w}_n = x_v^k$, where $v$ is an arbitrary vertex of $\mathbb{H}_n$:
\begin{equation}\label{eqxvk}
\bt_{\mathbb H_n} \approx \bt_{\mathbb H_n} + x_v^k \qquad (n \geq 2).
\end{equation}
\end{remark}


\begin{remark}
\cite[Theorem 4.1]{wzr} states that a finite flat semiring $S$ is nonfinitely based
if there exists $k \geq 3$ such that $S_c(a_1\cdots a_k) \in \mathsf{V}(S)$
and the condition $s^k = s_1\cdots s_k = 0$ forces $s = s_1 = \cdots = s_k = 0$ for all $s, s_1, \ldots, s_k \in S$.
This result is a corollary of Proposition~\ref{thnfb}; indeed, the finiteness assumption is not needed.
\end{remark}

It was shown in~\cite{jrz} that for any finite set of words $W$,
the semirings $M(W)$ and $M_c(W)$ are always nonfinitely based.
We now apply Proposition~\ref{thnfb} to generalize this result and obtain the following interval-type statement.
Let $\mathcal{V}$ be a variety and $\mathcal{W}$ a subvariety of $\mathcal{V}$.
The \emph{interval} $[\mathcal{W}, \mathcal{V}]$ denotes the set of all subvarieties of $\mathcal{V}$ that contain $\mathcal{W}$.

\begin{proposition}
Let $W$ be a finite nonempty set of words and let $k$ be an integer such that
$k > \ell(\mathbf{w})$ for every $\mathbf{w} \in W$.
Let $S$ denote the ai-semiring $M(W)$ or $M_c(W)$.
Then every variety in the interval $[\mathsf{V}(S_c(a_1\cdots a_k)), \mathsf{V}(S)]$ is nonfinitely based.
\end{proposition}

\begin{proof}
We first show that the interval $[\mathsf{V}(S_c(a_1\cdots a_k)), \mathsf{V}(S)]$ is well-defined.
To this end, we prove that $S_c(a_1\cdots a_k) \in \mathsf{V}(S)$.
Since $W$ is finite, choose $\mathbf{w} \in W$ of maximal length.
Then $\{0, \mathbf{w}, 1\}$ is a subsemiring of both $M(W)$ and $M_c(W)$, and is isomorphic to $S_7$.
By~\cite[Proposition 2.6]{jrz}, we have $S_c(a_1\cdots a_k) \in \mathsf{V}(S_7)$, hence $S_c(a_1\cdots a_k) \in \mathsf{V}(S)$.

Now let $\mathcal{V}$ be an arbitrary variety in $[\mathsf{V}(S_c(a_1\cdots a_k)), \mathsf{V}(S)]$.
To show that $\mathcal{V}$ is nonfinitely based,
it suffices, by Proposition~\ref{thnfb}, to prove that $S$ satisfies the identities~\eqref{eqxvk}.
Indeed, let $\varphi \colon \{x_v\mid v\in V(\mathbb H_n)\} \to S$ be an arbitrary assignment. Consider the following three cases.

\textbf{Case 1.} $\varphi(\bt_{\mathbb H_n}) = 0$. Then
\[
\varphi(\bt_{\mathbb H_n} + x_v^{k})
=\varphi(\bt_{\mathbb H_n}) + \varphi(x_v^{k})
=0+\varphi(x_v^{k})=0=\varphi(\bt_{\mathbb H_n}).
\]

\textbf{Case 2.} $\varphi(\bt_{\mathbb H_n}) = 1$. Then $\varphi(x_v)=1$ for all $v\in V(\mathbb{H}_n)$, and so
\[
\varphi(\bt_{\mathbb H_n} + x_v^{k})=\varphi(\bt_{\mathbb H_n}) + \varphi(x_v^{k})=1+1=1=\varphi(\bt_{\mathbb H_n}).
\]

\textbf{Case 3.} $\varphi(\bt_{\mathbb H_n}) = \mathbf{w}$ for some word $\mathbf{w}$ in $W^{\leq}$.
Then
\[
\varphi(x_{v_1})\varphi(x_{v_2})\cdots \varphi(x_{v_k}) = \mathbf{w}
\]
for every hyperedge $\{v_1, v_2, \ldots, v_k\} \in E(\mathbb{H}_n)$.
Since $\ell(\mathbf{w})<k$,
there exist indices $1 \leq i, j \leq k$ such that $\varphi(x_{v_i}) = 1$ and $\varphi(x_{v_j}) \neq 1$.
Now define a mapping $\psi \colon V(\mathbb{H}_n) \to \{1,2\}$ by
\[
\psi(v) =
\begin{cases}
1, & \text{if } \varphi(x_v) = 1,\\
2, & \text{if } \varphi(x_v) \neq 1.
\end{cases}
\]
This implies that $\mathbb{H}_n$ is $2$-colourable, contradicting the fact that $\mathbb{H}_n$ is not $2$-colourable.
So this case cannot occur.

Thus $S$ satisfies the identity~\eqref{eqwn}, and the desired conclusion follows from Proposition~\ref{thnfb}.
\end{proof}

A \emph{primitive} word is a word that is not a power of any word.
By Propositions~1.3.1 and~1.3.2 in~\cite{loth}, we have the following lemma.

\begin{lemma}\label{primitive}
Every word can be uniquely expressed as a power of a primitive word.
Moreover, a set of words is pairwise commutative if and only if they can be expressed as powers of the same primitive word.
\end{lemma}

The following result is due to Wu et al.~\cite[Proposition 2.7]{wzr}.
\begin{lemma}\label{lemma26072701}
Let $k \geq 1$ be an integer. Then $S_c(a_1 \cdots a_k) \in \mathsf{V}(S(a^{k+1}))$.
\end{lemma}

We now present a proposition that generalizes Lemma~\ref{lemma26072701}
and precisely characterizes when $S_c(a_1 \cdots a_k)$ lies in $\mathsf{V}(S(W))$.

\begin{proposition}\label{a1akak1}
Let $W$ be a nonempty set of words in $X^+$, and let $k \geq 2$ be an integer.
Then $S_c(a_1 \cdots a_k) \in \mathsf{V}(S(W))$ if and only if $W$  is not $x^{k+1}$-free.
\end{proposition}
\begin{proof}
Suppose that $W$ is not $x^{k+1}$-free.
Then there exists a word $\bw$ such that $\bw^{k+1} \in W^{\leq}$.
By Corollary~\ref{wwi}, $S(\bw^{k+1}) \in \mathsf{V}(S(W))$.
It is easy to see that $S(a^{k+1})$ is isomorphic to the subsemiring of $S(\bw^{k+1})$ generated by $\bw$.
Hence $S(a^{k+1}) \in \mathsf{V}(S(W))$.
By Lemma~\ref{lemma26072701}, $S_c(a_1 \cdots a_k) \in \mathsf{V}(S(a^{k+1}))$.
Therefore, $S_c(a_1 \cdots a_k) \in \mathsf{V}(S(W))$.

Conversely,
assume that $S_c(a_1 \cdots a_k) \in \mathsf{V}(S(W))$.
Let $\bu$ denote the polynomial
\[
\sum_{\sigma \in S_k} x_{{\sigma}(1)} x_{{\sigma}(2)} \cdots x_{{\sigma}(k)},
\]
where $S_k$ is the symmetric group of degree $k$.

Since $S_c(a_1 \cdots a_k)$ does not satisfy the identity $\bu \approx \bu+x_1^k$
(for instance, under the assignment $x_i \mapsto a_i$), neither does $S(W)$.
Hence there exists an assignment
$\varphi \colon \{x_1, \ldots, x_k\} \to S(W)$
such that $\varphi(\bu) \neq \varphi(\bu)+ \varphi(x_1^k)$,
which implies that $\varphi(\bu) \neq 0$.
Thus the words $\varphi(x_i) = \bw_{i} \in W^{\leq}$ satisfy
\begin{equation} \label{eq:sumsk}
\varphi(\bu) = \sum_{\sigma \in S_k} \bw_{{\sigma}(1)} \bw_{{\sigma}(2)} \cdots \bw_{{\sigma}(k)} \neq 0.
\end{equation}
Since $S(W)$ is a flat semiring, all summands of $\varphi(\bu)$ are nonzero and identical.
This implies that, for all $1 \leq i < j \leq k$,
\[
\bw_i \bw_j \cdot \prod_{\substack{1 \leq t \leq k \\ t \neq i,j}} \bw_t
=
\bw_j \bw_i \cdot \prod_{\substack{1 \leq t \leq k \\ t \neq i,j}} \bw_t.
\]
Thus $\bw_i \bw_j = \bw_j \bw_i$,
and so the words $\bw_1, \ldots, \bw_k$ commute pairwise.
By Lemma~\ref{primitive},
these words can be written as powers of a common primitive word $\bp$;
that is, for each $1 \leq i \leq k$, there exists $s_i \geq 1$ with $\varphi(x_i) = \bp^{s_i}$.
Then $\varphi(\bu) = \bp^s$, where $s = s_1 + \cdots + s_k \geq k$.

We claim that $s > k$; otherwise, if $s = k$, then $s_1 = \cdots = s_k = 1$,
and so $\varphi(x_1^k) = \bp^k=\bp^s$.
This implies that $\varphi(\bu) = \varphi(\bu)+ \varphi(x_1^k)$, a contradiction.
Thus $s > k$.
Since $W^{\leq}$ is closed under taking subwords and $\bp^s \in W^{\leq}$,
it follows that $\bp^{k+1} \in W^{\leq}$.
Therefore, $W$  is not $x^{k+1}$-free.
\end{proof}

\begin{theorem}\label{swnfb}
Let $W$ be a nonempty set of words in $X^+$ such that, for some $k \geq 3$,
$W$ is $x^{k+2}$-free, but not $x^{k+1}$-free.
Then every variety in the interval $[\mathsf{V}(S_c(a_1\cdots a_k)), \mathsf{V}(S(W))]$ is nonfinitely based.
In particular, $S(W)$ is nonfinitely based.
\end{theorem}

\begin{proof}
Since $W$ is not $x^{k+1}$-free, Proposition~\ref{a1akak1} implies that  $S_c(a_1 \cdots a_k)$ lies in the variety $\mathsf{V}(S(W))$. Hence, the interval $[\mathsf{V}(S_c(a_1\cdots a_k)), \mathsf{V}(S(W))]$ is well-defined.

Let $\mathcal{V}$ be an arbitrary variety in $[\mathsf{V}(S_c(a_1\cdots a_k)), \mathsf{V}(S(W))]$.
To show that $\mathcal{V}$ is nonfinitely based,
it suffices, by Proposition~\ref{thnfb} and Remark~\ref{remark:5.3}, to show that $S(W)$ satisfies the identities~\eqref{eqxvk}.
Let $\varphi \colon \{x_v\mid v\in V(\mathbb H_n)\} \to S(W)$ be an arbitrary assignment.
If $\varphi(\bt_{\mathbb H_n}) = 0$, then
\[
\varphi(\bt_{\mathbb H_n} + x_v^k)=\varphi(\bt_{\mathbb H_n})+\varphi(x_v^k)=0+\varphi(x_v^k)=0=\varphi(\bt_{\mathbb H_n}).
\]

Now suppose that $\varphi(\bt_{\mathbb H_n}) \neq 0$.
Then there exists a word $\mathbf{w} \in W^{\leq}$ such that
\[
\varphi(x_{v_1}) \varphi(x_{v_2}) \cdots \varphi(x_{v_k}) = \mathbf{w}
\]
for every hyperedge $e = \{v_1, \ldots, v_k\} \in E(\mathbb{H}_n)$.
Notice that for any pair of distinct indices $i$ and $j$, we have
\[
\varphi(x_{v_i}) \varphi(x_{v_j}) \bw_{ij} = \varphi(x_{v_j}) \varphi(x_{v_i}) \bw_{ij} = \mathbf{w},
\]
where $\bw_{ij} = \prod_{\substack{1 \leq t \leq k \\ t \neq i, j}} \varphi(x_{v_t})$.
By the $0$-cancellative law of $S(W)$, we obtain
\[
\varphi(x_{v_i}) \varphi(x_{v_j}) = \varphi(x_{v_j}) \varphi(x_{v_i}) \neq 0.
\]
By Lemma~\ref{primitive},
there exists a primitive word $\mathbf{p}_e \in W^{\leq}$ such that
for each $1 \leq i \leq k$ we have $r_i \geq 1$ with $\varphi(x_{v_i}) = \mathbf{p}_e^{r_i}$.
Thus $\mathbf{w} = \mathbf{p}_e^{t_e}$, where $t_e = r_1 + \cdots + r_k \geq k$.
Since every word in $W$ is $x^{k+2}$-free,
it follows that $t_e \in \{k, k+1\}$,
and so $\varphi(x_{v_i})$ is $\mathbf{p}_e$ or $\mathbf{p}_e^{2}$ for every $1\leq i \leq k$.
Similarly, for any hyperedge $f=\{w_1, \ldots, w_k\} \in E(\mathbb{H}_n)$,
there exist a word $\mathbf{p}_f \in W^{\leq}$ and $t_f \in \{k, k+1\}$ such that $\mathbf{w} = \mathbf{p}_f^{t_f}$,
and $\varphi(x_{w_i})$ is $\mathbf{p}_f$ or $\mathbf{p}_f^{2}$ for every $1\leq i \leq k$..
Note that $\mathbf{w} = \mathbf{p}_e^{t_e} = \mathbf{p}_f^{t_f}$, and $\mathbf{p}_e, \mathbf{p}_f$ are primitive.
By Lemma~\ref{primitive} again, $\mathbf{p}_e = \mathbf{p}_f$ and $t_e = t_f$.
Hence there exists a primitive word $\mathbf{p}$ and $t\in \{k, k+1\}$ such that
$\mathbf{w} = \mathbf{p}^t$ and $\varphi(x_v)$ is either $\mathbf{p}$ or $\mathbf{p}^2$ for every $v\in V(\mathbb{H}_n)$.

Consider the following two cases.

\textbf{Case 1.} $t = k$. Then $\varphi(x_v) = \mathbf{p}$ for every $x_v \in V(\mathbb{H}_n)$.
Thus $\varphi(x_v^k) = \varphi(\bt_{\mathbb H_n})=\mathbf{p}^k$,
and so
\[
\varphi(\bt_{\mathbb H_n} + x_v^k)=\varphi(\bt_{\mathbb H_n})+\varphi(x_v^k)
=\mathbf{p}^k+\mathbf{p}^k=\mathbf{p}^k=\varphi(\bt_{\mathbb H_n}).
\]

\textbf{Case 2.} $t = k+1$. Then, in the sense of a multiset,
\[
\{\varphi(x_{v_1}), \ldots, \varphi(x_{v_k})\} = \{\mathbf{p}^2, \mathbf{p}, \ldots, \mathbf{p}\}
\]
for every $\{v_1, v_2, \ldots, v_k\} \in E(\mathbb{H}_n)$.
This implies that $\mathbb{H}_n$ is $2$-colourable, contradicting the choice of $\mathbb{H}_n$.
So this case cannot occur.

Thus $S(W)$ satisfies the identities~\eqref{eqxvk}, and so the desired conclusion follows.
\end{proof}

\begin{remark}
Theorem~\ref{swnfb} generalizes \cite[Corollary 4.4]{wzr}, which states the following:
Let $W$ be a finite nonempty set of words in $X^+$ and let $k+1$ denote the number
$\max\{m\geq 1 \mid (\exists a \in X)~ a^m\in W^{\leq}\}$ with $k\geq 3$.
If $S(W)$ does not contain the $k$th power of any word of length more than $1$,
then it is nonfinitely based.

Every flat semiring $S(W)$ satisfying the conditions of \cite[Corollary 4.4]{wzr}
also satisfies the condition of Theorem~\ref{swnfb}.
Therefore, \cite[Corollary 4.4]{wzr} is a direct corollary of Theorem~\ref{swnfb}.
In fact, Theorem~\ref{swnfb} applies to flat semirings $S(W)$ that are not necessarily finite,
and its criterion is weaker;
thus it yields a stronger and more general result.

For example, let
\[
W' = \{x_1\cdots x_n \mid n \geq 1\} \cup \{x_1^4\}, \qquad
W'' = \{(xy)^4\}.
\]
Both $S(W')$ and $S(W'')$ satisfy the conditions of Theorem~\ref{swnfb} (take $k=5$),
so Theorem~\ref{swnfb} shows that they are both nonfinitely based.
In contrast, \cite[Corollary 4.4]{wzr} does not apply to either example.
\end{remark}

\begin{corollary}\label{scwk1}
Let $k \geq 3$ be an integer.
Then every variety in the interval $[\mathsf{V}(S_c(a_1a_2\cdots a_k)), \mathcal{W}_{k+1}]$ is nonfinitely based.
In particular, $\mathcal{W}_{k+1}$ is nonfinitely based.
\end{corollary}

\begin{proof}
Recall that $\mathcal{W}_{k+1}$ is generated by $S(W_{k+1})$,
where $W_{k+1}$ denotes the set of all words in $X^+$ of length $k+1$.
Every word in $W_{k+1}$ is $x^{k+2}$-free, but there is a word in $W_{k+1}$ is not $x^{k+1}$-free.
The required result now follows from Theorem~\ref{thnfb}.
\end{proof}

\begin{corollary}\label{wknfb}
Let $k \geq 1$ be an integer.
Then the variety $\mathcal{W}_k$ is finitely based if and only if $k \leq 3$.
\end{corollary}
\begin{proof}
This is a consequence of Corollaries~\ref{coro2026072610} and \ref{scwk1}.
\end{proof}

\begin{corollary}\label{ak1w}
Let $W$ be a finite nonempty set of words in $X^+$.
If $W$ is not $x^4$-free, then the flat semiring $S(W)$ is nonfinitely based.
\end{corollary}
\begin{proof}
Since $W$ is finite and contains a word that is not $x^4$-free,
there exists $k\geq 5$ such that every word in $W$ is $x^{k}$-free, but at least one word in $W$ is not $x^{k-1}$-free.
By Theorem~\ref{swnfb}, the required result holds.
\end{proof}

\subsection*{Acknowledgment}
Miaomiao Ren, corresponding author, is supported by National Natural Science Foundation of China (12371024, 12571020).
Xianzhong Zhao is supported by National Natural Science Foundation of China (12571020).

\section{Conclusion}
In this paper, we have studied the finite basis problem for flat semirings of the form $S(W)$,
where $W$ is not necessarily finite.
Our main results establish a sharp dichotomy: on the one hand,
$S(W)$ is finitely based (indeed, generates a Cross variety)
whenever every word in $W$ has length at most $3$ (Corollary~\ref{4swfb});
on the other hand, $S(W)$ is nonfinitely based whenever there exists $k \geq 3$ such that
$W$ is $x^{k+2}$-free but not $x^{k+1}$-free (Theorem~\ref{swnfb}).
As a consequence,
we obtain the complete classification: $\mathcal{W}_k$ is finitely based if and only if $k \leq 3$ (Corollary~\ref{wknfb}).
Moreover, $S(W)$ is nonfinitely based if $W$ is finite and not $x^4$-free (Corollary~\ref{ak1w}).
These results provide a partial solution to an open problem of Jackson et al.~\cite{jrz}.

It should be noted, however, that the condition in Theorem~\ref{swnfb} is not necessary.
For instance, \cite[Theorem 5.5]{rjzl} provides a family of nonfinitely based flat semirings of the form $S(W)$
in which every word in $W$ is $x^2$-free.
Specifically, for each integer $n>2$, define the word
\[
\boldsymbol{\ell}_n = x_1 \left( \prod_{i=1}^{n-1} x_{i+1} x_i \right) x_n.
\]
Then, whenever $M$ is an infinite subset of the set of integers greater than $2$ whose complement is also infinite,
the flat semiring $S(\boldsymbol{\ell}_n \mid n \in M)$ is nonfinitely based.
However, these algebras do not satisfy the condition of Theorem~\ref{swnfb}.

Based on Corollaries~\ref{4swfb} and \ref{ak1w}, we conjecture that the following results are true.

\begin{conjecture}
Let $W$ be a finite nonempty set of words in $X^+$.
Then the flat semiring $S(W)$ is finitely based if and only if
every word in $W$ has length at most $3$.
\end{conjecture}

The above conjecture naturally suggests a broader formulation, in which the finiteness assumption on $W$ is dropped:

\begin{conjecture}
Let $W$ be a nonempty set of words in $X^+$.
Then the flat semiring $S(W)$ is finitely based if and only if
every word in $W$ has length at most $3$.
\end{conjecture}

Resolving these conjectures would complete the classification of the finite basis problem for all flat semirings of the form $S(W)$.
The present paper has taken the first substantial step in this direction.

\end{document}